\documentclass[graybox]{svmult}

\usepackage{type1cm}
\usepackage{makeidx} 
\usepackage{graphicx}
\usepackage{multicol}
\usepackage[bottom]{footmisc}
\usepackage{mathtools}
\usepackage{newtxtext}
\usepackage{newtxmath}
\makeindex 
\usepackage{mathrsfs}
\usepackage{amsmath}
\usepackage{paralist, url}
\usepackage[colorlinks=true,hypertexnames=false]{hyperref}
\hypersetup{urlcolor=blue, citecolor=red}
\usepackage{enumitem}
\usepackage{colortbl}
\makeatletter\providecommand*{\toclevel@titlech}{0}\edef\toclevel@authorch{\the\numexpr\toclevel@titlech+1}\makeatother
\usepackage[T1]{fontenc}
\usepackage[utf8]{inputenc}
\usepackage{amsmath,amsxtra,bbm}
\usepackage{setspace,graphicx,color,pdflscape}
\usepackage[colorlinks=true]{hyperref}
\hypersetup{urlcolor=blue, linkcolor=blue, citecolor=red, anchorcolor=blue]}
\usepackage{bookmark} 
\usepackage{silence}

\newcommand{\R}{{\mathbb R}}

\newcommand{\be}[1]{\begin{equation}\label{#1}}
\newcommand{\ee}{\end{equation}}
\renewcommand{\(}{\left(}
\renewcommand{\)}{\right)}
\newcommand{\ird}[1]{\int_{\R^d}{#1}\,dx}
\newcommand{\nrm}[2]{\left\|{#1}\right\|_{#2}}
\newcommand{\nrmx}[2]{\left\|{#1}\right\|_{\mathrm L^{#2}(\R^d)}}
\newcommand{\bangle}[1]{\langle #1\rangle}
\newcommand{\irdx}[1]{\int_{\R^d}{#1}\,dx}
\newcommand{\irdv}[1]{\int_{\R^d}{#1}\,dv}
\newcommand{\irdxv}[1]{\iint_{\R^d\times\R^d}#1\,dx\,dv}
\newcommand{\nrmxv}[2]{\|#1\|_{\mathrm L^{#2}(\R^d\times\R^d,\,dx\,dv)}}
\newcommand{\nrmmu}[2]{\|#1\|_{\mathrm L^{#2}(\R^d\times\R^d,\,d\mu)}}
\newcommand{\nrmphi}[2]{\|#1\|_{\mathrm L^{#2}(\R^d,\,e^\phi\,dx)}}
\renewcommand{\qed}{\hfill\ $\square$}
\newcommand{\cpbx}[2]{\parbox{#1cm}{\begin{center}{#2}\end{center}}}

\makeatletter
\@namedef{subjclassname@2020}{%
 \textup{2020} Mathematics Subject Classification}
\makeatother
\newcommand{\msc}[1]{\href{https://mathscinet.ams.org/mathscinet/search/mscbrowse.html?sk=default&sk=#1&submit=Search}{#1}}

\definecolor{darkgreen}{rgb}{.0,.4,.2}

\title*{\texorpdfstring{$\mathrm L^2$}{L2} Hypocoercivity methods for kinetic Fokker-Planck equations with factorised Gibbs states}
\titlerunning{\texorpdfstring{$\mathrm L^2$}{L2} Hypocoercivity methods for kinetic Fokker-Planck equations with factorised Gibbs states}

\authorrunning{E.~Bouin, J.~Dolbeault, and L.~Ziviani}
\author{Emeric Bouin \and Jean Dolbeault \and Luca Ziviani}
\institute{Emeric Bouin \and Jean Dolbeault \and Luca Ziviani \at CEREMADE (CNRS UMR n$^\circ$ 7534), PSL university, Universit\'e Paris-Dauphine, Place de Lattre de Tassigny, 75775 Paris 16, France\\
\email{bouin@ceremade.dauphine.fr} (E.~Bouin), \email{dolbeaul@ceremade.dauphine.fr} (J.~Dolbeault), \email{luca.ziviani@dauphine.eu} (L.~Ziviani)\\
\url{https://www.ceremade.dauphine.fr/~bouin/} (E.~Bouin), \url{https://www.ceremade.dauphine.fr/~dolbeaul/} (J.~Dolbeault), \url{https://www.ceremade.dauphine.fr/fr/membres/detail-cv/profile/luca-ziviani.html} (L.~Ziviani)}

\begin{document}
\maketitle\vspace*{-2cm}
\thispagestyle{empty}


\abstract{This contribution deals with $\mathrm L^2$ hypocoercivity methods for kinetic Fokker-Planck equations with integrable local equilibria and a \emph{factorisation} property that relates the Fokker-Planck and the transport operators. Rates of convergence in presence of a global equilibrium, or decay rates otherwise, are estimated either by the corresponding rates in the diffusion limit, or by the rates of convergence to local equilibria, under moment conditions. On the basis of the underlying functional inequalities, we establish a classification of decay and convergence rates for large times, which includes for instance sub-exponential local equilibria and sub-exponential potentials.
\\[4pt]
}

\begin{flushright}\emph\today\end{flushright}
\newcommand{\nocontentsline}[3]{}
\newcommand{\tocless}[2]{\bgroup\let\addcontentsline=\nocontentsline#1{#2}\egroup}

\section{Introduction}\label{Sec:Introduction}

\emph{Hypocoercivity} refers to the method developed by C. Villani in order to capture large time asymptotics in kinetic equations, see~\cite{Mouhot-Neumann,Mem-villani}, which borrows ideas from H\"ormander's \emph{hypoelliptic} theory and from the \emph{carr\'e du champ} method introduced by D.~Bakry and M.~Emery in~\cite{Bakry1985}. For this reason, the Fisher information plays an important role and, to some extent, we can consider it as an $\mathrm H^1$-theory. Here we shall focus more on the notion of $\mathrm L^2$-hypocoercivity inspired by~\cite{MR2215889} and introduced in~\cite{Dolbeault2009511,DMS} in a simple case of a kinetic Fokker-Planck equation, which puts the emphasis on the underlying diffusion limit. The heuristic idea is simple: while the Fokker-Planck diffusion operator controls the rate of convergence towards local equilibria in the velocity space, the equilibration of the spatial density (its convergence to the spatial density of a global equilibrium or its decay when no such equilibrium exists) can be interpreted as a diffusion in the position space, at least in a certain parabolic scaling, which results of the interplay of the diffusion in the velocity direction and the transport and the mixing in the phase space induced by the transport operator. The advantage of this approach is that rates are fully determined by the functional inequalities associated to the diffusion operator on the velocity space and to the diffusion limit in the position space.

\newpage This paper is organized as follows. We start by recalling the abstract $\mathrm L^2$ hypocoercivity method in Theorem~\ref{ThmAbs:DMS} before applying it to the framework of non-Maxwellian local equilibria and a compatible transport operator in Corollary~\ref{Cor:DMS}. Although an adaptation of the standard theory in presence of microscopic and macroscopic coercivity associated respectively to the Fokker-Planck operator and the diffusion limit, this result is new and covers for instance the case of  relativistic transport. This framework is also well adapted to situations with weaker notions of coercivity corresponding to either an external potential with slower growth at infinity or to local equilibria with fatter tails than the Maxwellian. After reviewing various families of interpolation inequalities which dictate the asymptotic behaviours of the solutions to the Fokker-Planck equations, we extend the $\mathrm L^2$ hypocoercivity method to the kinetic Fokker-Planck equations and establish a classification in terms the (microscopic) local equilibria and the (macroscopic) equilibria associated with the Fokker-Planck diffusion limit.

\medskip Reviewing extensively the literature on the asymptotic behaviour of the solutions of Fokker-Planck, degenerate Fokker-Planck and kinetic Fokker-Planck equations goes beyond our scope. Let us simply quote some papers directly related to our methods, with more details to be given later.

Concerning \emph{Fokker-Planck} equations, coercivity for the diffusion operator means spectral gap and Poincar\'e inequality, and thus exponential decay of the solutions. This is a basic example of application of the carr\'e du champ method: see~\cite{MR3155209}. \emph{Weak Poincar\'e} inequalities are natural in the absence of spectral gaps as explained in~\cite{MR4056271} and have been quite systematically explored: see~\cite{MR1856277,MR2381160,MR4265692} and earlier references therein. However, such methods require strong assumptions on the initial data. This is the reason why we adopt an alternative approach based on moments and weighted functional inequalities, where the extreme case corresponds to Nash's inequality in absence of an external potential. See Table~\ref{Table:FP}.

As for \emph{kinetic Fokker-Planck} equations, hypocoercivity primarily refers to the method exposed in~\cite{Mem-villani}. We can also quote~\cite{piazzoli2019relaxation} for a detailed presentation of the commutator method and of Bakry-Emery type computations applied to estimates of the relaxation to equilibrium. More results and further references can also be found in~\cite{doi:10.1142/S0218202518500574}. In~\cite{MR3892316} S.~Hu and X.~Wang introduced a weak hypocoercivity approach \emph{\`a la} Villani, using a weak Poincar\'e inequality, and proved subexponential convergence to equilibrium. This was later extended to a class of degenerate diffusion processes in~\cite{MR4021241} by M.~Grothaus and F.-Y.~Wang using weak Poincar\'e inequalities for the symmetric and antisymmetric part of the generator, with non-exponential rates of convergence. In the same vein, C.~Cao proved quantitative convergence rates for the kinetic Fokker-Planck equation with more general confinement forces in~\cite{MR4069622,MR4278431}.

Alternatively the method of~\cite{DMS} was extended to cases without potentials in~\cite{BDMMS} while cases of weak or very weak potentials were considered respectively in~\cite{BDLS} and~\cite{BDS-very-weak}. Here the idea is to introduce moments and weighted interpolation inequalities to prove non-exponential decay or convergence rates. In these papers the effort has been mostly focused on the role of the external potential and fat tail local equilibria were not taken into consideration. However, it is known from~\cite{mellet_fractional_2011} that fat-tail local equilibria can be responsible of a fractional diffusion limit, which may govern decay rates in the case without external potential as shown in~\cite{Bouin_2022}, but this is not always the case. Non-Maxwellian local equilibria have been less explored than the Maxwellian case, but one has to mention~\cite{Brigati2023KRM, brigati2023construct} for such an extension of the earlier works~\cite{albritton2021variational,cao2023explicit}, with slightly different methods based on weak norms, Lions' lemma and time-averages.

\section{From microscopic and macroscopic coercivity to hypocoercivity}\label{Sec:Abstract}

Let us start by an expository section which collects some known results and introduces our present purpose. Let us consider the general evolution equation
\be{EqnEvol}
\frac{dF}{dt}+\mathsf TF=\mathsf LF
\ee
where $F$ is the density of a probability distribution defined on a real or complex Hilbert space $\mathscr H$ with scalar product $\bangle{\cdot,\cdot}$ and norm $\|\cdot\|$. We assume that $\mathsf T$ and $\mathsf L$ are two linear operators, respectively anti-Hermitian and Hermitian: $\mathsf T^*=-\,\mathsf T$ and $\mathsf \mathrm L^*=\mathsf L$, where ${}^*$ denotes the adjoint with respect to $\bangle{\cdot,\cdot}$. We are interested in the decay rate of $F$ or in the convergence to a steady state $F_\star$. We assume that $F_\star$ is unique, up to normalization. Since~\eqref{EqnEvol} is linear, we can always replace $F$ by $F-F_\star$ and study the convergence to $0$ of an eventually sign changing function $F$. We have in mind that $\mathsf L$ is an elliptic degenerate operator. If $\Pi$ is the orthogonal projection onto the kernel of $\mathsf L$, we assume that $\mathsf L$ has the \emph{microscopic coercivity} property in the sense that it is \emph{coercive} on $(1-\Pi)\,\mathscr H$, where $1$ is here a shorthand notation for the identity that will make sense in the functional setting of interest. In other words, we claim that
\begin{equation}
\frac12\,\frac d{dt}\,\|F\|^2=\langle\mathsf LF,F\rangle\le-\,\lambda_m\,\|(1-\Pi)F\|^2\label{H1}\tag{H1}
\end{equation}
for some $\lambda_m>0$. This is not enough to conclude that $\|F(t,\cdot)\|$ decays exponentially as we have no decay rate on $\mathrm{Ker}(\mathsf L)$, but if the operators $\mathsf L$ and $\mathsf T$ do not commute, we can hope that some of the decay properties on $(1-\Pi)\,\mathscr H$ are transferred on $\Pi\,\mathscr H$. This points towards the computation of various commutators and the whole machinery of H\"ormander's hypoellipticity theory. A micro/macro approach offers a simpler framework, that has the advantage of clarifying the role played by various functional inequalities in estimating decay rates of $F$. The underlying ideas rely on the formal \emph{macroscopic limit} of the scaled evolution equation
\[
\varepsilon\,\frac{dF}{dt}+\mathsf TF=\frac1\varepsilon\,\mathsf LF
\]
on the Hilbert space $\mathscr H$, which is a typical parabolic scaling when $\varepsilon$ is a small parameter. Using a formal expansion of a solution $F_\varepsilon=F_0+\varepsilon\,F_1+\varepsilon^2\,F_2+{\mathcal{O}(\varepsilon^3)}$ as $\varepsilon\to0_+$ and solving the equation order by order, we obtain
\[
\begin{array}{ll}
\mbox{at order}\,O\big(\varepsilon^{-1}\big):&\quad\mathsf LF_0=0\,,\\[4pt]
\mbox{at order}\,O\big(\varepsilon^0\big):&\quad\mathsf TF_0=\mathsf LF_1\,,\\[4pt]
\mbox{at order}\,O\big(\varepsilon^1\big):&\quad\frac{dF_0}{dt}+\mathsf TF_1=\mathsf LF_2\,.
\end{array}
\]
The first and second equation respectively read as $F_0=\Pi F_0$ and $F_1=-\,(\mathsf T\Pi)\,F_0$. After projection on $\mathrm{Ker}(\mathsf L)$, the third equation becomes $\frac d{dt}\(\Pi F_0\)-\,\Pi\mathsf T\,(\mathsf T\Pi)\,F_0=\Pi\mathsf LF_2=0$ that we can also write as
\be{Diffusion}
\frac{\partial F_0}{\partial t}+(\mathsf T\Pi)^*\,(\mathsf T\Pi)\,F_0=0\,.
\ee
Assuming \emph{macroscopic coercivity}, \emph{i.e.}, the property that the operator $(\mathsf T\Pi)^*\,(\mathsf T\Pi)$ is coercive on $(1-\Pi)\,\mathscr H$, we obtain
\begin{equation}
\frac12\,\frac d{dt}\,\|F_0\|^2=-\,\|(\mathsf T\Pi)\,F_0\|^2\le-\,\lambda_M\,\|F_0\|^2\label{H2}\tag{H2}
\end{equation}
for some $\lambda_M>0$. In order to derive~\eqref{Diffusion}, we implicitly used the fact that all terms are of order $\varepsilon$, which relies on the \emph{parabolic macroscopic dynamics} condition
\begin{equation}
\Pi\mathsf T\Pi\,F=0\,.\label{H3}\tag{H3}
\end{equation}
As in the \emph{hypocoercivity} method of~\cite{DMS}, let us consider the operator
\[\label{A}
\mathsf A:=\big(1+(\mathsf T\Pi)^*(\mathsf T\Pi)\big)^{-1}(\mathsf T\Pi)^*
\]
where the $(\mathsf T\Pi)^*(\mathsf T\Pi)$ term is of course reminiscent of~\eqref{Diffusion}, and, the Lyapunov functional, or \emph{entropy}, 
\begin{equation}\label{H}
\mathsf H[F]:=\frac12\,\|F\|^2+\delta\,\mathrm{Re}\bangle{\mathsf AF,F}\,.
\end{equation}
The parameter $\delta>0$, to be determined, as to be thought as a small parameter so that $\mathsf H[F]$ is a perturbation of $\frac12\,\|F\|^2$. The following estimate is by now classical but deserves some emphasis. Let us consider $G=\mathsf AF$, \emph{i.e.}, the solution of $(\mathsf T\Pi)^*F=G+(\mathsf T\Pi)^*\,\mathsf T\Pi\,G$. As in~\cite[Lemma~1]{DMS}, by a Cauchy-Schwarz inequality, we learn that
\begin{multline*}
\langle\mathsf{TA}F,F\rangle=\langle G,(\mathsf T\Pi)^*\,F\rangle=\|G\|^2+\|\mathsf T\Pi G\|^2=\|\mathsf AF\|^2+\|\mathsf{TA}F\|^2\\
\le\|\mathsf{TA}F\|\,\|(\mathrm{Id}-\Pi)F\|\le\frac1{2\,\mu}\,\|\mathsf{TA}F\|^2+\frac\mu2\,\|(\mathrm{Id}-\Pi)F\|^2\,.
\end{multline*}
Applied with either $\mu=1/2$ or $\mu=1$, this estimate proves that $\|\mathsf AF\|\le\frac12\,\|(\mathrm{Id}-\Pi)F\|$ and $\|\mathsf{TA}F\|\le\|(\mathrm{Id}-\Pi)F\|$.
Incidentally, this proves that
\begin{subequations}
\begin{align}
&\label{Est3}
\left|\langle\mathsf{TA}F,F\rangle\right|=\left|\langle\mathsf{TA}F,(\mathrm{Id}-\Pi)F\rangle\right|\le\|(\mathrm{Id}-\Pi)F\|^2\,,\\
&\label{Est4}
|\langle\mathsf AF,F\rangle|\le\frac12\,\|\Pi F\|\,\|(\mathrm{Id}-\Pi)F\|\le\frac14\,\|F\|^2\,.
\end{align}
\end{subequations}
We read from~\eqref{Est4} that $\mathsf H[F]$ and $\|F\|^2$ are equivalent with
\[
\frac{2-\delta}4\,\|F\|^2\le\mathsf H[F]\le\frac{2+\delta}4\,\|F\|^2\,.
\]
However, the twist introduced in $\mathsf H[F]$ by $\bangle{\mathsf AF,F}$ makes it exponentially decaying in $t$ if $F$ solves~\eqref{EqnEvol}. We can indeed compute
\[\label{D}
-\,\frac d{dt}\mathsf H[F]=\mathsf D[F]
\]
where
\be{D[F]}
\mathsf D[F]:=-\,\langle\mathsf LF,F\rangle+\delta\,\langle\mathsf{AT}\Pi F,F\rangle-\,\delta\,\Big(\mathrm{Re}\langle\mathsf{TA}F,F\rangle-\mathrm{Re}\langle\mathsf{AT}(1-\Pi)F,F\rangle+\mathrm{Re}\langle\mathsf{AL}F,F\rangle\Big)\,.
\ee
By~\eqref{H1}, we know that $-\,\langle\mathsf LF,F\rangle\ge\lambda_m\,\|(1-\Pi)F\|^2$. On the other hand,~\eqref{H2} amounts to
\[
\big\langle(\mathsf T\Pi)^*\,(\mathsf T\Pi)\,F,F\big\rangle\ge\lambda_M\,\|F\|^2\quad\mbox{if}\quad F\in\mathrm{Ker}(\mathsf L)
\]
and, by construction, the operator $\mathsf A$ is therefore such that
\be{gap}
\langle\mathsf{AT}\Pi F,F\rangle\ge\frac{\lambda_M}{1+\lambda_M}\,\|\Pi F\|^2\,.
\ee
The first two terms in the definition of $\mathsf D[F]$ can be combined to prove that
\be{Dprod}
\mathsf D[F]\ge\lambda_m\,\|(1-\Pi)F\|^2+\frac{\delta\,\lambda_M}{1+\lambda_M}\,\|\Pi F\|^2-\,\delta\,\Big(\mathrm{Re}\langle\mathsf{TA}F,F\rangle-\mathrm{Re}\langle\mathsf{AT}(1-\Pi)F,F\rangle+\mathrm{Re}\langle\mathsf{AL}F,F\rangle\Big)\,.
\ee
Under the additional assumption that the last term in the above identity involves only \emph{bounded auxiliary operators} in the sense that 
\begin{equation}
\|\mathsf{AT}(1-\Pi)F\|+\|\mathsf{AL}F\|\le C_M\,\|(1-\Pi)F\|\,,\label{H4}\tag{H4}
\end{equation}
one obtains the \emph{entropy -- entropy production} inequality
\[
\mathsf D[F]\ge\lambda\,\mathsf H[F]
\]
for some explicit constant $\lambda>0$. The precise statement goes as follows. It has been established in~\cite{DMS} in the case of a real Hilbert space $\mathscr H$ and extended to complex Hilbert spaces in~\cite{BDMMS}.
\begin{theorem}[\cite{BDMMS,DMS}]\label{ThmAbs:DMS} Let $\mathsf L$ and $\mathsf T$ be closed linear operators in the complex Hilbert space $\big(\mathcal H,\langle\cdot,\cdot\rangle\big)$. We assume that $\mathsf L$ is Hermitian and $\mathsf T$ is anti-Hermitian, and that~\eqref{H1}--\eqref{H4} hold for some positive constants $\lambda_m$, $\lambda_M$, and~$C_M$. Then there is some $\delta_\star\in(0,2)\cap(0,\lambda_m)$ such that, for any $\delta\in(0,\delta_\star)$, there are explicit constants $\lambda>0$ and $\mathscr C>1$ for which, if $F$ solves~\eqref{EqnEvol} with initial datum $F_0\in\mathcal H$, then
\be{Decay:BDMMS}
\mathsf H[F(t,\cdot)]\le\mathsf H[F_0]\,e^{-\lambda\,t}\quad\mbox{and}\quad\|F(t,\cdot)\|^2\le\mathscr C\,e^{-\lambda\,t}\,\|F_0\|^2\quad\forall\,t\ge0\,.
\ee
\end{theorem}
The estimates of $\lambda>0$ and $\mathscr C>1$ in~\cite[Proposition~4]{BDMMS} have been improved in~\cite[Proposition~2]{Arnold2021} as follows. With $X:=\|(\mathrm{Id}-\Pi)F\|$ and $Y:=\|\Pi F\|$. Using~\eqref{Est4}, we read from~\eqref{H} that
\[
\mathsf H[F]\le\frac12\(X^2+Y^2\)+\frac\delta2\,X\,Y
\]
while it follows from~\eqref{Dprod},~\eqref{Est3} and~\eqref{H4} that
\[
\mathsf D[F]-\lambda\,\mathsf H[F]\ge\(\lambda_m-\,\delta-\frac\lambda2\)X^2-\,\delta\(C_M+\frac\lambda2\)X\,Y+\(\frac{\delta\,\lambda_M}{1+\lambda_M}-\frac\lambda2\)Y^2\,.
\]
With $K_M:=\frac{\lambda_M}{1+\lambda_M}<1$ and $\delta_\star:=\frac{4\,K_M\,\lambda_m}{4\,K_M+C_M^2}<\lambda_m$, a simple discriminant condition shows that for any $\delta\in(0,\delta_\star)$, the right-hand side is nonnegative for the largest (positive) solution of
\[
\delta^2\(C_M+\frac\lambda2\)^2-4\(\lambda_m-\,\delta-\frac\lambda2\)\(\frac{\delta\,\lambda_M}{1+\lambda_M}-\frac\lambda2\)=0\,.
\]
We refer to~\cite{Arnold2021} for further details and to~\cite{MR3522857} for more considerations on the functional framework.

\medskip In the framework of kinetic equations, $\mathsf T$ and $\mathsf L$ are respectively the \emph{transport operator} and the \emph{collision operator} acting on a distribution function $f(t,x,v)$ where $t\ge0$ is the time, $x$ is the position and $v$ is the velocity. To fix ideas, we shall assume that $x$, $v\in\R^d$ and consider
\begin{itemize}
\item a transport operator defined by the Poisson bracket as
\be{Tkin}
\mathsf Tf:=\nabla_v\mathscr E\cdot\nabla_xf-\nabla_x\mathscr E\cdot\nabla_vf
\ee
corresponding to the Hamiltonian energy
\[
(x,v)\mapsto\mathscr E(x,v):=\frac1\beta\,\bangle v^\beta+\phi(x)\,,
\]
where $\phi$ denotes an external, given potential,
\item a collision operator of Fokker-Planck type given by
\be{Lkin}
\mathsf Lf:=\nabla_v \cdot \(\nabla_vf+v\,\bangle v^{\beta-2}\,f\)\,.
\ee
\end{itemize}
Here we use the notation
\[
\bangle v:=\sqrt{1+|v|^2}\,.
\]
Unless $\beta=2$, our choice of the transport operator differs from the transport operator corresponding to Newton's equations, namely $v\cdot\nabla_x-\nabla_x\phi\cdot\nabla_v$, which has been widely studied in the literature. More general dependences of $\mathcal E$ and $\mathsf L$ on $v$ than $\bangle v^\beta$, with for instance a power law asymptotic growth as $|v|\to+\infty$, could be considered under minor changes. Our purpose is to study the asymptotic behaviour of the solution of
\be{kFP}
\frac{\partial f}{\partial t}+\mathsf Tf=\mathsf Lf\,,\quad f(t=0,\cdot,\cdot)=f_0
\ee
with $\mathsf T$ and $\mathsf L$ given respectively by~\eqref{Tkin} and~\eqref{Lkin} as $t\to+\infty$. With these choices and under the condition that $e^{-\phi}$ is integrable, a remarkable property is that the \emph{Gibbs state}
\be{fstar}
f_\star(x,v):=\frac1Z\,e^{-\mathscr E(x,v)}\quad\mbox{where}\quad Z=\irdx{e^{-\phi}}\irdv{e^{-\frac1\beta\,\bangle v^\beta}}
\ee
is a stationary solution of mass $\nrmxv{f_\star}1=1$. We consider $\mathsf L$ as an operator on $\mathrm L^2\big(\R^d,e^{\bangle v^\beta}dv\big)$ acting on functions depending on the velocity variable $v$ and extend it to the Hilbert space $\mathrm L^2(\R^d\times\R^d,d\mu)$ of functions depending on $x$ and $v$ where
\[
d\mu:=\frac{dx\,dv}{f_\star(x,v)}\,.
\]
Since $f_\star$ is integrable, notice that $\mathrm L^1(\R^d\times\R^d,dx\,dv)\subset\mathrm L^2\big(\R^d\times\R^d,d\mu)$ by a Cauchy-Schwarz inequality. After these preliminaries, we observe that $f_\star$ is local equilibrium, \emph{i.e.}, belongs to $\mathrm{Ker}(\mathsf L)$ which is generated by functions of the type
\be{urho}
f_\rho(x,v):=\frac{\rho(x)\,e^{-\frac1\beta\,\bangle v^\beta}}{\irdv{e^{-\frac1\beta\,\bangle v^\beta}}}
\ee
where $\rho\in\mathrm L^2\big(\R^d,e^{-\phi}\,dx\big)\supset\mathrm L^1(\R^d,dx)$ is an arbitrary function. The property $\mathsf Tf_\star=\mathsf Lf_\star=0$ sometimes appears in the physics literature as a \emph{factorization} property. The orthogonal projector onto $\mathrm{Ker}(\mathsf L)$ is defined as the projection on local equilibria by
\[
\Pi f=f_\rho(x,v)\quad\mbox{where}\quad\rho=\irdv f\,.
\]
Notice that $f$ and $f_\rho$ have the same spatial density because $\irdv{f_\rho}=\rho$. Under the assumption that the measure $e^{-\phi}\,dx$ admits a Poincar\'e inequality, that is, there is some positive constant $\lambda_\phi$ for which
\be{Poincare}
\irdx{|\nabla u|^2\,e^{-\phi}}\ge\lambda_\phi\irdx{|u|^2\,e^{-\phi}}\quad\forall\,u\in\mathscr D(\R^d)\quad\mbox{such that}\quad\irdx{u\,e^{-\phi}}=0\,,
\ee
Theorem~\ref{ThmAbs:DMS} applies as follows.
\begin{corollary}\label{Cor:DMS} Assume that $\phi$ is such that~\eqref{Poincare} holds for some $\lambda_1>0$ and $\beta\ge1$. If $f$ solves~\eqref{kFP} for some nonnegative function $f_0\in\mathrm L^2(\R^d\times\R^d,d\mu)$ with $\nrmxv{f_0}1=1$, then for some $\delta>0$, there exists $\lambda>0$ and $\mathscr C>1$ such that~\eqref{Decay:BDMMS} holds with $\|\cdot\|:=\nrmmu\cdot2$. 
\end{corollary}
The case $\beta=2$ is by now standard and covered in various papers: see~\cite{Dolbeault2009511,DMS} for an $\mathrm L^2$ hypocoercivity approach and~\cite{MR2034753,MR2130405,MR2215889} as well as references therein for earlier results based on hypoelliptic methods. To our knowledge the case $\beta\neq2$ has not been studied yet by $\mathrm L^2$-hypocoercivity methods, but convergence results are known from~\cite{cao2023explicit,brigati2023construct} using other methods. Of particular interest in physics is the case $\beta=1$ where $\mathscr E$ is the standard energy for relativistic particles, up to physical constants (mass and speed of light are taken equal to $1$), while the corresponding $\mathsf L$ operator is not much more than a caricature of a relativistic collision operator. Concerning $\mathsf L$ and from the point of view of phenomenological models, it is however interesting to consider local equilibria given by~\eqref{urho} and it makes sense to assume that stationary solutions have the \emph{factorization} property. Throughout this paper, we will make this simplifying assumption.

The strategy of the proof of Corollary~\ref{Cor:DMS} goes as follows. With $\beta\ge1$, we have the Poincar\'e inequality: there is some $\lambda_m>0$ such that, for all $g\in\mathrm L^2(\R^d,e^{-\frac1\beta\,\bangle v^\beta}\,dv)$ such that $\irdv{g\,e^{-\frac1\beta\,\bangle v^\beta}}=0$, we have
\be{Poincare:v}
\irdv{|\nabla g|^2\,e^{-\frac1\beta\,\bangle v^\beta}}\ge\lambda_m\irdv{|g|^2\,e^{-\frac1\beta\,\bangle v^\beta}}\,.
\ee
This is for instance a consequence of Persson's lemma based on the observation that $\psi(v):=\frac1\beta\,\bangle v^\beta$ is such that
\[
\liminf_{|v|\to+\infty}\(\frac14\,|\nabla\psi(v)|^2-\frac12\,\Delta\psi(v)\)>0\,.
\]
See for instance~\cite[Appendix~A.1]{Bouin_2020} for details. As a consequence,~\eqref{H1} holds. The macroscopic coercivity condition~\eqref{H2} follows from~\eqref{Poincare}. The \emph{parabolic macroscopic dynamics} condition~\eqref{H3} is a simple consequence of the definitions of $\mathsf T$ and $\Pi$. Hence the only assumption that deserves some attention is~\eqref{H4}, which is obtained by elliptic estimates. A detailed proof is given in Section~\ref{Sec:ThmMain}.

\medskip In the framework of~\eqref{Tkin} and~\eqref{Lkin}, an elementary computation shows that~\eqref{Diffusion} written for $f_\rho$ defined by~\eqref{urho} reduces to the \emph{Fokker-Planck} equation
\be{FP}
\frac{\partial\rho}{\partial t}=\sigma\,\big(\Delta\rho+\nabla\cdot(\rho\,\nabla\phi)\big)
\ee
with diffusion coefficient $\sigma$ given in terms of $\beta$ by
\be{sigma}
\sigma=\frac1d\irdv{|v|^2\,\bangle v^{2\beta-4}\,e^{-\frac1\beta\,\bangle v^\beta}}\,.
\ee
In order to simplify the discussion, we shall assume that
\[
\phi(x)=\frac1\alpha\,\bangle x^\alpha\quad\forall\,x\in\R^d\,.
\]
Corollary~\ref{Cor:DMS} corresponds to $\alpha\ge1$. Our purpose is to investigate the decay rates of~\eqref{kFP} in terms of $\beta>0$ and $\alpha>0$. Let us start by studying the asymptotic behaviour of a solution $\rho$ of~\eqref{FP} depending on the various cases for the potential $\phi$. For completeness, we will also consider the limit case as $\alpha\to0$ and distinguish several cases depending on whether we take $\phi=0$ a.e., or (in the case of the Fokker-Planck equation), \hbox{depending on
~$\gamma>0$},
\[
\phi(x)=\gamma\,\log\bangle x\quad\forall\,x\in\R^d\,.
\]
Up to minor technicalities, general potentials $\phi$ with asymptotic power law or logarithmic growths as $|x|\to+\infty$ could also be covered.

\section{Fokker-Planck equations with various external potentials, moments and functional inequalities}\label{Sec:Diffusion}

We collect some results on the asymptotic behaviour of the solutions of~\eqref{FP} as $t\to+\infty$ based on various functional inequalities. 
In this section we omit the discussion of optimality cases and estimates on sharp constants in the functional inequalities. By default, constants in the inequalities are always taken to their optimal value. Table~\ref{Table:FP} collects the results in a synthetic picture, although without all details on the assumptions.

\subsection{Strong confinement case: Poincar\'e inequality}
If $\phi(x)=\frac1\alpha\,\bangle x^\alpha$ with $\alpha\ge1$, then~\eqref{Poincare} holds with $\lambda_\phi=\lambda_M>0$. We apply it to $u=\rho/e^{-\phi}$. A solution $\rho$ of~\eqref{FP} with initial datum $\rho_0$ at $t=0$ satisfies
\[
\frac d{dt}\nrmphi{\rho(t,\cdot)}2^2=-\,2\,\sigma\,\nrmphi{\nabla\rho(t,\cdot)}2^2\le-\,2\,\lambda_M\,\sigma\,\nrmphi{\rho(t,\cdot)}2^2\,,
\]
which yields the estimate
\[
\nrmphi{\rho(t,\cdot)}2^2\le\nrmphi{\rho_0}2^2\,e^{-2\,\lambda_M\,\sigma\,t}\quad\forall\,t\ge0\,.
\]

\subsection{Weak confinement case: weighted Poincar\'e inequality}

The following results are taken from~\cite[Appendices~A and~B]{BDLS}. We assume here that $\alpha\in(0,1)$ and consider a solution of~\eqref{FP} with nonnegative initial datum $\rho_0\in\mathrm L^1(\R^d,dx)$ such that $\nrmx{\rho_0}1=1$. The function $u=\rho\,e^\phi$ is a solution of the \emph{Ornstein-Uhlenbeck} equation (also known as the \emph{backward Kolmogorov} equation)
\be{Eq:OU}
\frac{\partial u}{\partial t}=\sigma\,\big(\Delta u-\nabla \phi\cdot\nabla u\big)\,.
\ee
With $k\ge0$, let us compute
\[
\frac d{dt}\irdx{|u(t,x)|^2\,\bangle x^k\,e^{-\phi}}+2\,\sigma\irdx{|\nabla_xu(t,x)|^2\,\bangle x^k\,e^{-\phi}}\le\irdx{\big(a_k-b_k\,\bangle x^{\alpha-2}\big)\,|u(t,x)|^2\,\bangle x^k\,e^{-\phi}}
\]
for some $a_k\in\R$, $b_k\in(0,+\infty)$. As a consequence, there exists a constant $\mathcal K(k)>0$ such that
\[
\irdx{\bangle x^k\,|\rho(t,x)|^2\,e^\phi}\le\mathcal K(k)\irdx{\bangle x^k\,|\rho_0|^2\,e^\phi}\quad\forall\,t\ge0\,.
\]
See~\cite[Proposition~4 and Appendix~B.2]{BDLS} for details. With $k=0$, we notice that $a_0=b_0=0$ and use the \emph{weighted Poincar\'e inequality}
\be{wPoincaré:x}
\irdx{|\nabla_xu(t,x)|^2\,e^{-\phi}}\ge\mathscr C_\alpha^{\mathrm{wP}}\irdx{|u(t,x)-\bar u|^2\,\frac{e^{-\phi}}{\bangle x^{2\,(1-\alpha)}}}\quad\mbox{where}\quad\bar u=\frac{\irdx{u\,e^{-\phi}}}{\irdx{e^{-\phi}}}
\ee
(notice that the average $\bar u$ is computed with respect to the measure of the l.h.s.) to prove that
\[
\frac d{dt}\irdx{|u(t,x)-\bar u|^2\,e^{-\phi}}=-\,2\,\sigma\irdx{|\nabla_xu(t,x)|^2\,e^{-\phi}}\le-\,2\,\sigma\,\mathscr C_\alpha^{\mathrm{wP}}\irdx{|u(t,x)-\bar u|^2\,\frac{e^{-\phi}}{\bangle x^{2\,(1-\alpha)}}}\,.
\]
With $k\ge2\,(1-\alpha)$ and $\theta=k/\big(k+2\,(1-\alpha)\big)$, H\"older's inequality
\[
\irdx{|u-\bar u|^2\,e^{-\phi}}\le\(\irdx{|u-\bar u|^2\,\frac{e^{-\phi}}{\bangle x^{2\,(1-\alpha)}}}\)^\theta\(\irdx{|u-\bar u|^2\,\bangle x^k\,e^{-\phi}}\)^{1-\theta}
\]
allows us to prove that
\[
\irdx{|\rho(t,x)-\rho_\star(x)|^2\,e^\phi}\le\(\nrmphi{\rho_0-\rho_\star}2^{-\,4\,(1-\alpha)/k}+\frac{4\,(1-\alpha)\,\sigma\,\mathscr C_\alpha^{\mathrm{wP}}}{k\,\mathcal K_*^{\,2\,(1-\alpha)/k}}\,t\)^{-\frac k{2\,(1-\alpha)}}\quad\forall\,t\ge0
\]
with $\mathcal K_*:=\mathcal K(k)^2\,\irdx{\bangle x^k\,|\rho_0|^2\,e^\phi}+\irdx{\bangle x^k\,e^{-\phi}}\nrmx{\rho_0}1^2$.

\subsection{Weak confinement, a limit case: Hardy-Poincar\'e inequality}

The results in this case are new. In the limit as $\alpha\to0_+$, we can assume that $\phi(x)=\gamma\,\log\bangle x$ with $\gamma>d$ so that $f_\star$ defined by~\eqref{fstar} is integrable. Let $u=\rho\,e^\phi$ be a solution of~\eqref{Eq:OU}. With $k\ge0$, let us compute
\begin{multline}\label{eq:dissipL2}
\frac d{dt}\irdx{|u(t,x)|^2\,\bangle x^k\,e^{-\phi}}+2\,\sigma\irdx{|\nabla_xu(t,x)|^2\,\bangle x^k\,e^{-\phi}}\\
=k\irdx{|u(t,x)|^2\,\bangle x^{k-2}\(d+(k-\gamma-2)\,\tfrac{|x|^2}{\bangle x^2}\)e^{-\phi}}\\
\le k\irdx{|u(t,x)|^2\,\bangle x^{k-2}\(d+k-\gamma-2-(k-\gamma-2)\,\bangle x^{-2}\)e^{-\phi}}\,.
\end{multline}
By arguing as in~\cite[Proposition~4 and Appendix~B.2]{BDLS}, this is enough to prove that there exists a constant $\mathcal K(k)>0$ such that
\[
\irdx{\bangle x^k\,|u(t,x)|^2\,e^{-\phi}}\le\mathcal K(k)\irdx{\bangle x^k\,|\rho_0|^2\,e^\phi}\quad\forall\,t\ge0
\]
if $k\in(\gamma-d,\gamma+2-d)$. Notice that a better range of $k$ can be obtained as follows. Since $\bangle x^k\,e^{-\phi}=\bangle x^{k-\gamma}$, we learn from~\cite{MR2609957,MR2823993} that for some positive constant $\mathscr C^{\mathrm{HP}}_{\gamma-k}$, we have the \emph{Hardy-Poincar\'e} inequality
\be{HP}
\irdx{|\nabla_xu|^2\,\bangle x^k\,e^{-\phi}}\ge\mathscr C^{\mathrm{HP}}_{\gamma-k}\irdx{|u-\bar u|^2\,\bangle x^{k-2}\,e^{-\phi}}
\ee
for an appropriate choice of $\bar u$ depending on $k-\gamma$. In any case, Inequality~\eqref{eq:dissipL2} written with $k=0$, that is,
\[
\frac d{dt}\irdx{|u(t,x)-\bar u|^2\,e^{-\phi}}\le-\,2\,\sigma\irdx{|\nabla_xu(t,x)|^2\,e^{-\phi}}
\]
and then~\eqref{HP} combined with H\"older's inequality applied as in the case of the weighted Poincar\'e inequality (with $\alpha=0$) show that
\[
\irdx{|\rho(t,x)-\rho_\star(x)|^2\,e^\phi}\le\irdx{|\rho_0-\rho_\star|^2\,e^\phi}\,(1+c\,t)^{-\frac k2}\quad\forall\,t\ge0
\]
for some constant $c$ which depends on $d$, $\gamma$, $\sigma$, $k$, $\irdx{|\rho_0|^2\,\bangle x^{k-\gamma}}$ and $\nrmx{\rho_0}1^2$.

\subsection{Very weak confinement case: Caffarelli-Kohn-Nirenberg inequality}

According to~\cite[Theorem~1]{BDS-very-weak}, if $1\le\gamma<d$ and $\phi(x)=\gamma\,\log\bangle x$, a solution $\rho$ of~\eqref{FP} with nonnegative initial datum $\rho_0\in\mathrm L^1(\R^d,\bangle x^k\,dx)\cap\mathrm L^2\big(\R^d,e^\phi\,dx\big)$ with $k=\max\{2,\gamma/2\}$ satisfies the estimate
\[
M_k(t):=\ird{\bangle x^k\,\rho(t,x)}\le2^\frac{k-2}2\(M_0+\(\big(M_k(0)-M_0\big)^{2/k}+2\,\sigma\,\big(d+k-2-\gamma\big)\,M_0^{2/k}t\)^{k/2}\)\,.
\]
With $e^{-\phi}=\bangle x^{-\gamma}$ and $u=\rho\,\bangle x^\gamma$, a solution $u$ of~\eqref{Eq:OU} satisfies the estimate
\[
\frac d{dt}\irdx{|u(t,x)|^2\,\bangle x^{-\gamma}}=-\,2\,\sigma\irdx{|\nabla u(t,x)|^2\,\bangle x^{-\gamma}}
\]
Combined with the \emph{inhomogeneous Caffarelli-Kohn-Nirenberg inequality}
\[
\ird{|u|^2\,\bangle x^{-\gamma}}\le\mathscr C_{k,\gamma}^{\mathrm{CKN}}\(\ird{|\nabla u|^2\,\bangle x^{-\gamma}}\)^a\(\ird{u\,\bangle x^{k-\gamma}}\)^{2(1-a)}\quad\mbox{with}\quad a=\frac{d+2k-\gamma}{d+2+2k-\gamma}\,,
\]
this proves the decay estimate
\[
\nrmphi{\rho(t,\cdot)}2^2\le\nrmphi{\rho_0}2^2\,(1+c\,t)^{-\frac{d-\gamma}2}\quad\forall\,t\ge0
\]
where the constant $c$ depends on $d$, $\gamma$, $\sigma$, $\nrmphi{\rho_0}2$, $M_0=\nrm{u_0}1$, and $M_k(0)=\nrm{|x|^k\,\rho_0}1$. For more details, as well as a proof of the Caffarelli-Kohn-Nirenberg inequality, see~\cite[Appendix~B]{BDS-very-weak}.
\subsection{No potential case: Nash's inequality}

We assume that $\phi=0$ so that~\eqref{FP} is the standard heat equation. By \emph{Nash's inequality}
\[\label{Nash}
\nrmx u2\le\mathscr C_{\mathrm{Nash}}\,\nrmx{\nabla u}2^\frac d{d+2}\,\nrmx u1^\frac2{d+2}\quad\forall\,u\in\mathrm H^1(\R^d,dx)\,,
\]
a solution $\rho$ of~\eqref{FP} with initial datum $\rho_0$ at $t=0$ satisfies
\[
\frac d{dt}\nrmx{\rho(t,\cdot)}2^2=-\,2\,\sigma\,\nrmx{\nabla\rho(t,\cdot)}2^2\,.
\]
Hence $y(t):=\nrmx{\rho(t,\cdot)}2^2$ solves the differential inequality $y'\le-\,2\,\sigma\,\mathscr C_{\mathrm{Nash}}^{-1}\,\nrmx{\rho_0}1^{-\frac4d}\,y^{1+\frac2d}$ which, after integration, yields the estimate
\[
\nrmx{\rho(t,\cdot)}2^2\le\Big(\nrmx{\rho_0}2^{-4/d}+\,\tfrac{4\,\sigma}{d\,\mathscr C_{\mathrm{Nash}}}\,\nrmx{\rho_0}1^{-4/d}\,t\Big)^{-d/2}\quad\forall\,t\ge0\,.
\]

\vspace*{-0.5cm}
\begin{table}[hb]
\begin{center}
\begin{tabular}{|l|c|c||c|c|c|}
\hline
\cpbx{1.8}{Potential}&\cpbx{2.5}{$\phi=0$}&\cpbx{2.5}{$\phi(x)=\gamma\,\log\bangle x$\\$\gamma<d$}&\cpbx{2.5}{$\phi(x)=\gamma\,\log\bangle x$\\$\gamma>d$}
&\cpbx{2.5}{$\phi(x)=\frac1\alpha\,\bangle x^\alpha$\\$\alpha\in(0,1)$}&\cpbx{2.5}{$\phi(x)=\frac1\alpha\,\bangle x^\alpha$\\$\alpha\ge1$}\\
\hline
\cpbx{1.8}{Inequality}&\cellcolor[gray]{0.9}Nash&\cpbx{2.5}{Caffarelli-Kohn-\\Nirenberg}&\cpbx{2.5}{Hardy-Poincar\'e}
&\cellcolor[gray]{0.9}\cpbx{2.5}{Weighted Poincar\'e}&\cellcolor[gray]{0.9}Poincar\'e\\
\hline
\cpbx{1.8}{Asymptotic\\behavior}&\cellcolor[gray]{0.9}\cpbx{2.5}{$t^{-d/2}$\\decay}&\cpbx{2.5}{$t^{-(d-\gamma)/2}$\\decay}&\cpbx{2.5}{$t^{-k/2}$\\convergence}
&\cellcolor[gray]{0.9}\cpbx{2.5}{$t^{-\frac k{2\,(1-\alpha)}}$\\convergence}&\cellcolor[gray]{0.9}\cpbx{2.5}{$e^{-\lambda\,t}$\\convergence}\\
\hline
\cpbx{1.8}{References}&\cpbx{2.5}{\cite{Nash58}}&\cpbx{2.5}{\cite{BDS-very-weak}}&\cpbx{2.5}{}
&\cpbx{2.5}{\cite{BDLS}}&\cpbx{2.5}{$(*)$}\\
\hline
\end{tabular}
\caption{\label{Table:FP} Short summary of the behaviours as $t\to+\infty$ of the solution of~\eqref{FP} depending on the choice of $\phi$, with some references. On the left side ($\phi=0$ or $\gamma<d$), there is no global stationary solution and we study decay rates. On the right side, we investigate the convergence rates to a global stationary solution. Under additional or different constraints on the initial data, other behaviours can be obtained based for instance on \emph{weak Poincar\'e inequalities}: see~\cite[Theorem~2.1]{MR1856277},~\cite[Theorem~1.4]{MR2381160} and~\cite{MR4265692}.\newline $(*)$ The use of the Poincar\'e inequality in relation with the Fokker-Planck equation has a long history, which we cannot cover entirely here: we can for instance refer to~\cite{Nash58}, and to~\cite[Chapter~4]{MR3155209} for an overview in the context of Markov processes.}
\end{center}
\end{table}
\clearpage
\subsection{A short summary}

In case of the Fokker-Planck equation~\eqref{FP}, Table~\ref{Table:FP} summarizes what is known on decay rates based on moment estimates and interpolation inequalities. Cases in gray will be further considered in the case of kinetic~equations.

\section{Kinetic Fokker-Planck equations and hypocoercivity results}\label{Sec:Hypocoercivity}

\subsection{State of the art}

Some known results are collected in Table~\ref{Table:kFP}. They are exclusively concerned with the classical transport operator
\[
\mathsf Tf:=v\cdot\nabla_xf-\nabla_x\phi\cdot\nabla_vf\,,
\]
\emph{i.e.}, coincide with our framework if $\beta=2$ (at the level of the transport operator).
\begin{table}
\begin{center}
\begin{tabular}{|l|c||c|c|}
\hline
\cpbx{2.5}{Potential}&\cpbx{2.85}{$\phi=0$}
&\cpbx{2.5}{$\phi(x)=\frac1\alpha\,\bangle x^\alpha$\\$\alpha\in(0,1)$}&\cpbx{2.5}{~\\$\phi(x)=\frac1\alpha\,\bangle x^\alpha$\\$\alpha\ge1$, or $\mathbb T^d$\\ Macro Poincar\'e\\}\\
\hline
\cpbx{2.5}{$\psi(v)=\frac1\beta\,\bangle v^\beta$\\ $\beta\ge1$\\Micro Poincar\'e}&\cellcolor[gray]{0.9}\cpbx{2.85}{$t^{-d/2}$\\decay\\\cite{BDMMS}}
&\cellcolor[gray]{0.9}\cpbx{2.5}{$e^{-t^b}$,~$b<1$\\ $\beta=2$\\convergence\\~\cite{MR4069622}}&\cellcolor[gray]{0.9}\cpbx{2.5}{$e^{-\lambda\,t}$\\convergence\\~\cite{MR2215889,Mouhot-Neumann,Dolbeault2009511,DMS,MR3488535,albritton2021variational,cao2023explicit,Brigati2023KRM,brigati2023construct}}\\\hline
\cpbx{2.5}{$\psi(v)=\frac1\beta\,\bangle v^\beta$\\ $\beta\in(0,1)$}&\cellcolor[gray]{0.9}\cpbx{2.85}{$t^{-\zeta}$\\ $\zeta=\min\{\frac d2,\frac\ell{2(1-\beta)}\}$\\decay,~\cite{BDLS}}
&\cellcolor[gray]{0.9}\cpbx{2.5}{$t^{-\zeta}$\\ convergence\\\cite{BDZ-2}}&\cellcolor[gray]{0.9}\cpbx{2.5}{$t^{-\zeta}$\\ convergence\\\cite{BDZ-2}}\\\hline
\cpbx{2.5}{Limit as $\beta\to 0_+$\\ $\psi(v)=-(d+\varepsilon)\,\log\bangle v$}&\cpbx{2.5}{$\varepsilon\in(0,2)$\\ fractional dif-\\fusion limit,~\cite{Bouin_2022}}
&\cite{BDZ-2}&\cpbx{2.5}{$t^{-\zeta}$ if $\varepsilon>2$\\ convergence\\\cite{BDZ-2}}\\\hline
\end{tabular}
\caption{\label{Table:kFP} Rough classification of the asymptotic behaviour of the solutions of $\partial_tf+v\cdot\nabla_xf-\nabla_x\phi\cdot\nabla_vf=f_\star\,\nabla_v\big(f_\star^{-1}\,\nabla_vf\big)$ as $t\to+\infty$ where $f_\star(x,v)=Z^{-1}\,\exp\(-\phi(x)-\psi(v)\)$. Additional assumptions on the initial datum $f_0=f(t=0,\cdot,\cdot)$ are needed: for instance in the case $\alpha\ge1$ and $\beta\in(0,1)$, the initial datum is such that $f_0\in\mathrm L^2\big(\R^d\times\R^d,\bangle v^{(1-\beta)\,\sigma}\,d\mu\big)$. In the case $\phi=0$ and $\beta\in(0,1)$, we assume that $f_0\in\mathrm L^2\big(\R^d\times\R^d,\bangle v^{\ell/2}\,d\mu\big)$. If $\alpha\in(0,1)$ and $\beta\in(0,1)$, using the weak Poincar\'e inequality requires specific bounds. Further cases and more detailed assumptions can be found in the references collected above.}
\end{center}
\end{table}

In Table~\ref{Table:kFP}, if $\beta\ge1$, \emph{Micro} Poincar\'e refers to a Poincar\'e inequality written in the velocity variable $v$, which controls the convergence towards a local equilibrium while, \emph{Macro} Poincar\'e refers to a Poincar\'e inequality written in the position variable $x$, which controls the convergence of the solution in the macroscopic or diffusion limit, towards a global equilibrium, or to $0$ if there is no such equilibrium. Cases in gray will be further considered in the case of the transport operator given by~\eqref{Tkin}.

\subsection{Notation and basic observations}\label{Sec:Notation}

{}From here on, we assume that $\psi(v)=\frac1\beta\,\bangle v^\beta$ and $\phi(x)=\frac1\alpha\,\bangle x^\alpha$ for some $\beta>0$ and $\alpha>0$, use the notation
\[
\rho_\star:=\frac{e^{-\phi}}{\irdx{e^{-\phi}}}\,,\quad\rho_f:=\irdv f\quad\mbox{and}\quad u_f:=\frac{\rho_f}{\rho_\star}\,,
\]
and consider the transport operator given by~\eqref{Tkin}. We recall that $f_\star$ is defined by~\eqref{fstar}. The following observations can be omitted at first reading and will be used only in Sections~\ref{Sec:EntropyProduction}--\ref{Sec:ThmMain} for proving Theorem~\ref{Thm:Main}. We can write
\[
\Pi f=\rho_f\,\frac{e^{-\psi}}{\irdv{e^{-\psi}}}=u_f\,f_\star\,,\quad\mathsf T\Pi f=\bangle v^{\beta-2}\,(v\cdot\nabla_xu_f)\,f_\star\quad\mbox{and}\quad\Pi\mathsf Tf=\(\nabla_x\cdot\irdv{v\,\bangle v^{\beta-2}\,f}\)\frac{f_\star}{\rho_\star}\,.
\]
If $f=u\,f_\star\in\mathrm{Ker}(\mathsf L)$ and $\sigma$ is defined by~\eqref{sigma}, then
\[
(\mathsf T\Pi)^*\,(\mathsf T\Pi)\,f=-\,\frac{\sigma}{\rho_\star}\,\nabla_x\cdot\big(\rho_\star\,\nabla_x u\big)\,f_\star=-\,\sigma\,\big(\Delta_xu-\nabla_x\phi\cdot\nabla_xu\big)\,f_\star\,,
\]
Solving $g=\big(1+(\mathsf T\Pi)^*(\mathsf T\Pi)\big)^{-1}f$ means that $g=u\,f_\star$ where $u=u_g$ solves
\be{ell-u}
u-\sigma\,\big(\Delta_xu-\nabla_x\phi\cdot\nabla_xu\big)=u_f\,.
\ee
In order to justify integrations by parts (see Section~\ref{Sec:EntropyProduction} below), one can notice that $c\,f_\star$ with an arbitrary $c\in\R$ can be used as a barrier function, so that we can assume that $u_f$ is bounded as $|x|\to+\infty$. Standard elliptic estimates apply to the solution $u$ of~\eqref{ell-u} and one can conclude using density arguments.

\subsection{Main result}

Our goal is to get a classification similar to the results summarized in Table~\ref{Table:kFP} for $\beta\neq2$ in the transport operator defined by~\eqref{Tkin}, \emph{i.e.}, with $\mathsf Tf:=\nabla_v\mathscr E\cdot\nabla_xf-\nabla_x\mathscr E\cdot\nabla_vf$. As far as we know, this transport operator has not been studied yet in the framework of hypocoercivity methods, except for some recent results in the bounded domain case in~\cite{albritton2021variational,Brigati2023KRM} or when $\alpha\ge1$ in~\cite{cao2023explicit,brigati2023construct} which are based on weak norms and Lions' lemma.
\begin{theorem}\label{Thm:Main}
Let $f=f(t,x,v)$ be a solution of~\eqref{kFP} with transport and collision operators given respectively by~\eqref{Tkin} and~\eqref{Lkin} for some $\beta>0$ and $\alpha>0$. With $f_\star$ defined by~\eqref{fstar}, we assume that the initial datum satisfies
\be{Linfty}
0\le f_0\le C\,f_\star
\ee
for a suitable constant $C>0$. Depending on $\beta$ and $\alpha$, we have the following convergence and decay estimates.
\begin{enumerate}
\item Assume $\beta\ge 1$ and $\alpha\ge 1$. Then there exist constants $\mathscr C>0$ and $\lambda>0$ such that any solution $f$ of~\eqref{kFP} with initial datum $f_0\in \mathrm L^2(\R^d\times\R^d,d\mu)$ satisfies 
\[
\nrmmu{f-f_\star}2^2\le \mathscr C e^{-\lambda t}\,\nrmmu{f_0-f_\star}2^2\quad\forall\,t\ge0\,.
\]
\item Assume $\beta\in(0,1)$ and $\alpha\ge 1$. Then there exists a constant $\mathscr C_\ell>0$ such that any solution $f$ of~\eqref{kFP}, with initial datum $f_0\in \mathrm L^2\big(\R^d\times\R^d,\bangle v^\ell\,d\mu\big)$ for some $\ell>0$, satisfies
\[
\nrmmu{f-f_\star}2^2\le \mathscr C_\ell\,(1+t)^{-\frac\ell{2(1-\beta)}}\,\nrmmu{f_0-f_\star}2^2\quad\forall\,t\ge0\,.
\]
\item Assume $\beta\ge 1$ and $\alpha\in(0,1)$. Then there exists a constant $\mathscr C_k>0$ such that any solution $f$ of~\eqref{kFP}, with initial datum $f_0\in \mathrm L^2\big(\R^d\times\R^d,\bangle x^k\,d\mu\big)$ for some $k>0$, satisfies
\[
\nrmmu{f-f_\star}2^2\le \mathscr C_k\,(1+t)^{-\frac k{2(1-\alpha)}}\,\nrmmu{f_0-f_\star}2^2\quad\forall\,t\ge0\,.
\]
\item Assume $\beta\in(0,1)$ and $\alpha\in(0,1)$. Then there exist a constant $\mathscr C_{k,\ell}>0$ such that any solution $f$ of~\eqref{kFP}, with initial datum $f_0\in \mathrm L^2\big(\R^d\times\R^d,\bangle x^k\,d\mu\big)\cap \mathrm L^2\big(\R^d\times\R^d,\bangle v^\ell\,d\mu\big)$ for some $k>0$ and $\ell>0$, satisfies
\[
\nrmmu{f-f_\star}2^2\le \mathscr C_{k,\ell}\,(1+t)^{-\zeta}\,\nrmmu{f_0-f_\star}2^2\quad\forall\,t\ge0\,.
\]
where $\zeta=\min\big\{\frac k{2(1-\alpha)},\frac\ell{2(1-\beta)}\big\}$
\item Assume $\beta \ge 1$ and $\phi=0$. Then there exist a constant $\mathscr K>0$ depending on $\nrmxv{f_0}1$ such that any solution $f$ of~\eqref{kFP}, with initial datum $f_0\in \mathrm L^2(\R^d\times\R^d,d\mu)$, satisfies
\[
\nrmmu f2^2\le \mathscr K\,(1+t)^{-\frac d2}\,\nrmmu{f_0}2^2\quad\forall\,t\ge0\,.
\]
\item Assume $\beta\in(0,1)$ and $\phi=0$. Then there exist a constant $\mathscr K_\ell>0$ such that any solution $f$ of~\eqref{kFP}, with initial datum $f_0\in \mathrm L^2\big(\R^d\times\R^d,\bangle v^\ell\,d\mu\big)$ for some $\ell>0$, satisfies
\[
\nrmmu f2^2\le \mathscr K_\ell\,(1+t)^{-\zeta}\,\nrmmu{f_0}2^2\quad\forall\,t\ge0\,,
\]
where $\zeta=\min\big\{\frac d2,\frac\ell{2(1-\beta)}\big\}$.
\end{enumerate}
\end{theorem}
In the statement of Theorem~\ref{Thm:Main}, even if it is not specified, the constants may depend on norms of $f_0$. See Table~\ref{Table:kFP2} for a summary of the results. Assumption~\eqref{Linfty} is a simplifying assumption which can be removed in various cases: see for instance~\cite{DMS,MR3488535,MR4069622,BDLS}. It allows an immediate conservation of moments along the flow, see Lemma~\ref{Lem:UniformBounds} below.

\begin{table}
\begin{center}
\begin{tabular}{|l|c||c|c|}
\hline
\cpbx{2.5}{Potential}&\cpbx{2.5}{$\phi=0$}&\cpbx{2.5}{$\phi(x)=\frac1\alpha\,\bangle x^\alpha$\\$\alpha\in(0,1)$}&\cpbx{2.75}{~\\$\phi(x)=\frac1\alpha\,\bangle x^\alpha$\\$\alpha\ge1$}\\
\hline
\cpbx{2.5}{$\psi(v)=\frac1\beta\,\bangle v^\beta$\\ $\beta\ge1$\\Micro Poincar\'e}&\cellcolor[gray]{0.9}\cpbx{2.5}{$t^{-d/2}$\\decay}
&\cellcolor[gray]{0.9}\cpbx{2.75}{$t^{-\frac k{2(1-\alpha)}}$\\ convergence}&\cellcolor[gray]{0.9}\cpbx{2.5}{$e^{-\lambda t}$\\convergence}\\\hline
\cpbx{2.5}{$\psi(v)=\frac1\beta\,\bangle v^\beta$\\ $\beta\in(0,1)$}&\cellcolor[gray]{0.9}\cpbx{2.5}{$t^{-\min\big\{\frac d2,\frac\ell{2(1-\beta)}\big\}}$\\convergence}
&\cellcolor[gray]{0.9}\cpbx{2.75}{$t^{-\min\big\{\frac k{2(1-\alpha)},\frac\ell{2(1-\beta)}\big\}}$\\ convergence}&\cellcolor[gray]{0.9}\cpbx{2.5}{$t^{-\frac \ell{2(1-\beta)}}$\\ convergence}\\\hline
\end{tabular}
\caption{\label{Table:kFP2} Summary of the results of Theorem~\ref{Thm:Main}. See the statement for the precise meaning of the rates and the assumptions.}
\end{center}
\end{table}
\noindent\emph{Remark.}
The results of Theorem~\ref{Thm:Main} can be extended to functions $\psi$ and $\phi$ depending monotonously on $|v|$ and~$|x|$ respectively, which behave like $\bangle{v}^\beta$ and $\bangle{x}^\alpha$ as $|v|\to+\infty$ and $|x|\to+\infty$. Typically, one has to assume that for any $v\in\R^d$,
\[
C_1\,\bangle{v}^\beta\leq \phi(v)\leq C_2\,\bangle{v}^\beta\,,\quad C_3\,|v|\,\bangle{v}^{\beta-1}\leq v\cdot\nabla_v\psi(v)\leq C_4\,|v|\,\bangle{v}^{\beta-1}\quad\mbox{and}\quad|\mathrm{Hess}(\psi)|(v)\leq C_5\,\bangle{v}^{\beta-2}
\]
for some positive constants $C_i$, with $i=1,\dots,5$, and similar estimates for $\phi$.

\subsection{An estimate of the entropy production}\label{Sec:EntropyProduction}

Let us introduce the weighted norm defined by
\[
\|f\|_\beta:=\|f\|_{\mathrm L^2(\R^d\times\R^d,\,\bangle v^{-2(1-\beta)_+}\,d\mu)}
\]
where $(1-\beta)_+$ denotes the positive part of $1-\beta$. As a consequence $\|f\|_2$ denotes the standard norm with no weight and we keep using the notation $\bangle{\cdot,\cdot}$ for the associated scalar product. We can rephrase the Poincar\'e inequality~\eqref{Poincare:v} corresponding to the case $\beta\ge1$, and the weighted Poincar\'e inequality~\eqref{wPoincaré:x} rewritten in the variable $v$ with $\beta\in(0,1)$ instead of $\alpha$ and $\lambda_m=\mathscr C_\beta$, as
\[
-\,\bangle{\mathsf Lf,f}\ge\lambda_m\,\|(1-\Pi)f\|_\beta^2\,.
\]
In the language of Theorem~\ref{ThmAbs:DMS}, this inequality replaces~\eqref{H1} while~\eqref{H3} is still satified. Next we use the notation of Section~\ref{Sec:Abstract} for $\mathsf A$, $\mathsf H$ and~$\mathsf D$, with $\mathsf T$ and $\mathsf L$ given respectively by~\eqref{Tkin} and~\eqref{Lkin}. 
\begin{lemma}\label{LemmaDissip} For any $\beta>0$, there is a positive constant $\kappa$ such that
\be{Estimate:Dissipation}
\mathsf D[f]\ge\kappa \(\|(1-\Pi)f\|_\beta^2+\bangle{\mathsf{AT}\Pi f,\Pi f}\)\,.
\ee
\end{lemma}
\begin{proof} We recall that by~\eqref{D[F]}, $\mathsf D$ is defined as
\[
\mathsf D[f]:=-\,\langle\mathsf Lf,f\rangle+\delta\,\langle\mathsf{AT}\Pi f,f\rangle-\,\delta\,\Big(\mathrm{Re}\langle\mathsf{TA}f,f\rangle-\mathrm{Re}\langle\mathsf{AT}(1-\Pi)f,f\rangle+\mathrm{Re}\langle\mathsf{AL}f,f\rangle\Big)\,.
\]
In order to prove Lemma~\ref{LemmaDissip}, we have to give estimates on the last three terms using~$\|(1-\Pi)f\|_\beta$ and $\bangle{\mathsf{AT}\Pi f,\Pi f}$. We obtain these estimates  in four steps, as follows.

\medskip\noindent\textit{\textbf{Step 1.} Expressions of $\bangle{\mathsf{AT}\Pi f,\Pi f}$.} We consider the function $u=u(x)$ implicitly defined by $u\,f_\star=\big(1+(\mathsf T\Pi)^\ast(\mathsf T\Pi)\big)^{-1}\,\Pi f$, that is, the solution of~\eqref{ell-u} that can be rewritten as
\be{u}
u-\frac\sigma{\rho_\star}\,\nabla_x\cdot\big(\rho_\star\,\nabla_xu\big)=u_f=\frac{\rho_f}{\rho_\star}\,.
\ee
We deduce from
\[
\mathsf{AT}\Pi f=\big(1+(\mathsf T\Pi)^\ast(\mathsf T\Pi)\big)(\mathsf T\Pi)^\ast(\mathsf T\Pi)f=\Pi f-\big(1+(\mathsf T\Pi)^\ast(\mathsf T\Pi)\big)^{-1}f=\Pi f-u\,f_\star=\(\frac{\rho_f}{\rho_\star}-u\)f_\star
\]
and~\eqref{u} that
\[
\bangle{\mathsf{AT}\Pi f,\Pi f}=-\,\sigma\irdx{\big(\nabla_x\cdot\big(\rho_\star\,\nabla_xu\big)\big)\,\rho_f}=-\,\sigma\irdx{\nabla_x\cdot\big(\rho_\star\,\nabla_xu\big)\(u-\frac\sigma{\rho_\star}\,\nabla_x\cdot\big(\rho_\star\,\nabla_xu\big)\)}\,,
\]
that is,
\be{ATPi}
\bangle{\mathsf{AT}\Pi f,\Pi f}=\sigma\irdx{|\nabla_x u|^2\,\rho_\star}+\sigma^2\irdx{\big|\nabla_x\cdot\big(\rho_\star\,\nabla_xu\big)\big|^2\,\rho_\star^{-1}}\,.
\ee
Testing~\eqref{u} with $u\,\rho_\star$, we learn after an integration by parts that
\[
\ird{|u|^2\,\rho_\star}+\sigma\ird{|\nabla u|^2\,\rho_\star}=\ird{u\,\rho_f}\le\(\ird{|u|^2\,\rho_\star}\)^{1/2}\,\|\Pi f\|
\]
where the inequality arises from a Cauchy-Schwarz estimate using $\|\Pi f\|^2=\ird{|\rho_f|^2\,\rho_\star^{-1}}$. Hence
\be{Test1}
\ird{|u|^2\,\rho_\star}\le\|\Pi f\|^2\quad\mbox{and}\quad\sigma\ird{|\nabla u|^2\,\rho_\star}\le\|\Pi f\|^2\,.
\ee

\medskip\noindent\textit{\textbf{Step 2.} An estimate of $|\bangle{\mathsf{TA} f,f}|$.} We know from~\eqref{Est3} that $|\bangle{\mathsf{TA} f,f}|\le\|(1-\Pi)f\|_2^2$ if $\beta\ge1$. With $\sigma$ defined by~\eqref{sigma}, we claim that
\be{Estm:TA}
\big|\mathrm{Re}\bangle{\mathsf{TA} f,f}\big|\le\frac1\sigma\,\|(1-\Pi)f\|_\beta^2
\ee
also holds if $\beta\in(0,1)$. In this later case, let us consider the function $w=w(x)$ implicitly defined by $w\,f_\star=\mathsf A f$, that is, the solution of 
\[
w-\frac\sigma{\rho_\star}\,\nabla_x\cdot(\rho_\star\,\nabla_xw)=-\,\frac1{\rho_\star}\nabla_x\cdot\irdv{v\,\bangle v^{\beta-2}f}\,.
\]
Testing with $w\,\rho_\star$ we obtain
\[
\sigma\irdx{|\nabla_xw|^2\,\rho_\star}\le\irdx{|w|^2\,\rho_\star}+\sigma\irdx{|{\nabla_xw}|^2\,\rho_\star}=\irdx{\nabla_xw\cdot\(\irdv{v\,\bangle v^{\beta-2}f}\)}\,.
\]
By applying the Cauchy-Schwarz inequality and after squaring, we obtain
\[
\sigma^2\irdx{|\nabla_xw|^2\,\rho_\star}\le\irdx{\(\irdv{v\,\bangle v^{\beta-2}(1-\Pi)f}\)^2\,\rho_\star^{-1}}\le\|(1-\Pi)f\|_\beta^2
\]
using $|v|/\bangle v\le1$ so that $\irdv{|v|^2\,\bangle v^{-2}\,e^{-\psi}}/\irdv{e^{-\psi}}\le1$. Altogether, we prove~\eqref{Estm:TA} with
\[
\|\mathsf{TA}f\,\bangle v^{(1-\beta)_+}\,\|_2^2=\|\mathsf T(w\,f_\star)\,\bangle v^{(1-\beta)_+}\,\|_2^2=\irdxv{\big|\tfrac v{\bangle v}\cdot\nabla_xw\big|^2\,\bangle v^{2(\beta-1)_+}\,f_\star}\le\frac1{\sigma^2}\,\|(1-\Pi)f\|_\beta^2\,.
\]

\medskip\noindent\textit{\textbf{Step 3.}} We claim that, for some explicit constant $C_\beta>0$,
\be{ALf}
\big|\mathrm{Re}\bangle{\mathsf{AL}(1-\Pi)f,f}\big|\le\frac{C_\beta}\sigma\,\|(1-\Pi)f\|_\beta\,\bangle{\mathsf{AT}\Pi f,\Pi f}\,.
\ee
This follows from a direct computation. Consider $u=u(x)$ defined by~\eqref{u} and observe that $\mathsf A^\ast\Pi f= \mathsf T(u\,f_\star)$. Then $\bangle{\mathsf{AL}(1-\Pi)f,\Pi f}=\big\langle(1-\Pi)f,\mathsf L \mathsf T(u\,f_\star)\big\rangle$. Since 
\[
\mathsf L\mathsf T(u\,f_\star)=\mathsf L\(\bangle v^{\beta-2}\,v\cdot\nabla_xu\,f_\star\)=\big(\xi(v)\cdot\nabla_xu\big)\,f_\star
\]
where $\xi(v):=\nabla_v\cdot\mathrm{Hess}_v(\psi)-\nabla_v \psi\cdot\mathrm{Hess}_v(\psi)$ is a vector valued function of $v$, it follows that
\[ 
\|\mathsf L\mathsf T(u\,f_\star)\bangle v^{(1-\beta)_+}\,\|^2_2\le C_\beta\irdx{|\nabla_xu|^2\,\rho_\star}\quad\mbox{with}\quad C_\beta=\irdv{\frac{|\xi(v)|^2\,\bangle v^{2(1-\beta)_+}\,e^{-\psi}}{\irdv{e^{-\psi}}}}\,.
\]
Then~\eqref{ALf} follows from~\eqref{ATPi}.

\medskip\noindent\textit{\textbf{Step 4.}} We claim that, for some explicit constant $C>0$,
\be{Ellip}
\big|\bangle{\mathsf{AT}(1-\Pi)f,\Pi f}\big|\le C\,\|(1-\Pi)f\|_\beta\,\bangle{\mathsf{AT}\Pi f,\Pi f}^{1/2}.
\ee
With $u=u(x)$ defined by~\eqref{u}, we have $(\mathsf{AT})^\ast\Pi f=-\,\mathsf T^2(u\,f_\star)$ and
\[
\bangle{\mathsf{AT}(1-\Pi)f, \Pi f}=\bangle{(1-\Pi)f,-\mathsf T^2(u\,f_\star)}\,.
\]
Using the expression~\eqref{Tkin} for $\mathsf{T}$, a computation yields 
\[
\mathsf T^2(u\,f_\star)=\Big(\nabla_v\psi\cdot \mathrm{Hess}_x(u)\cdot\nabla_v\psi-\nabla_xu\cdot\mathrm{Hess}_v(\psi)\cdot\nabla_x\phi\Big)\,f_\star\,.
\]
Adding and subtracting $(\nabla_v\psi\cdot\nabla_xu)\,(\nabla_v\psi\cdot\nabla_x\phi)$, we can rewrite
\[
\mathsf{T}^2(uf_\star)=\nabla_v\psi\cdot\big(\mathrm{Hess}(u)-\nabla_xu\otimes\nabla_x\phi\big)\cdot\nabla_v\psi\,f_\star-\nabla_xu\cdot\big(\mathrm{Hess}(\psi)-\nabla_v\psi\otimes\nabla_v\psi\big)\cdot\nabla_x\phi\,f_\star
\]
where $\nabla_xu\otimes\nabla_x\phi$ and $\nabla_v\psi\otimes\nabla_v\psi$ are respectively the matrices with entries $\big(\partial_{x_i}u\,\partial_{x_j}\phi\big)_{i,j}$ and $\big(\partial_{v_i}\psi\,\partial_{v_j}\psi\big)_{i,j}$\,. We estimate independently the two terms in the expression of $\mathsf{T}^2(uf_\star)$.
\\[4pt]
(1) The second term is estimated by
\begin{align*}
\big\|\nabla_xu\cdot\big(\mathrm{Hess}(\psi)-\nabla_v\psi\otimes\nabla_v\psi\big)&\cdot\nabla_x\phi\,f_\star\,\bangle v^{(1-\beta)_+}\big\|_2^2\\
&\le\irdxv{|\nabla_xu|^2\,\big|\,\mathrm{Hess}(\psi)-\nabla_v\psi\otimes\nabla_v\psi\big|^2\,|\nabla_x\phi|^2\,\bangle v^{2(1-\beta)_+}\,f_\star}\\
&\le C_{\beta,2}\irdx{|\nabla_xu|^2\,|\nabla_x\phi|^2\,\rho_\star}
\end{align*}
where $C_{\beta,2}=\irdv{\big|\,\mathrm{Hess}(\psi)-\nabla_v\psi\otimes\nabla_v\psi\big|^2\,\bangle v^{2(1-\beta)_+}\,e^{-\psi}}\big/\irdv{e^{-\psi}}$. Using the fact that $|\nabla_x\phi|^2$ is bounded for $\alpha\in(0,1)$ and~\cite[Lemma 8]{DMS} if $\alpha\ge1$, there is some constant $c_\alpha>0$ such that the solution of~\eqref{u} satisfies
\be{ImprPoinc}
\irdx{|\nabla_xu|^2\,|\nabla_x\phi|^2\,\rho_\star}\le c_\alpha\irdx{|\nabla_xu|^2\,\rho_\star}\,,
\ee
which, after using~\eqref{ATPi}, is enough to obtain the bound
\[
\big\|\nabla_xu\cdot\big(\mathrm{Hess}(\psi)-\nabla_v\psi\otimes\nabla_v\psi\big)\cdot\nabla_x\phi\,f_\star\,\bangle v^{(1-\beta)_+}\big\|_2^2\le\frac{c_\alpha\,C_{\beta,2}}\sigma\,\|(1-\Pi)f\|_\beta\,\bangle{\mathsf{AT}\Pi f,\Pi f}^{1/2}\,.
\]
(2) For the first term, we have
\begin{align*}
\big\|\nabla_v\psi\cdot\big(\mathrm{Hess}(u)-\nabla_xu\otimes\nabla_x\phi\big)&\cdot\nabla_v\psi\,f_\star\,\bangle v^{(1-\beta)_+}\big\|^2_2\\
&\le\irdxv{|\nabla_v\psi|^4\,\big|\,\mathrm{Hess}(u)-\nabla_xu\otimes\nabla_x\phi\big|^2\,\bangle v^{2(1-\beta)_+}\,f_\star}\\
&\le C_{\beta,3}\irdx{\big|\,\mathrm{Hess}(u)-\nabla_xu\otimes\nabla_x\phi\big|^2\,\rho_\star}
\end{align*}
where $C_{\beta,3}=\irdv{|\nabla_v\psi|^4\,\bangle v^{2(1-\beta)_+}\,e^{-\psi}}\big/\irdv{e^{-\psi}}$. Notice that $\mathrm{Hess}(u)-\nabla_xu\otimes\nabla_x\phi$ is the matrix with entries $\partial_{x_ix_j}u-\partial_{x_i}u\,\partial_{x_j}\phi=\partial_{x_i}\big(\rho_\star\,\partial_{x_j}u\big)\,\rho_\star^{-1}$ for $i,j=1,\dots,d$. Hence
\begin{multline}\label{ClainedIdentity}
\int_{\R^d}\big|\,\mathrm{Hess}(u)-\nabla_xu\otimes\nabla_x\phi\big|^2\,\rho_\star\,dx=\sum_{i,j=1}^d\irdx{\(\partial_{x_i}\big(\rho_\star\,\partial_{x_j}u\big)\)^2\,\rho_\star^{-1}}\\
=\irdx{\big|\nabla_x\cdot\big(\rho_\star\,\nabla_xu\big)\big|^2\,\rho_\star^{-1}}+\irdx{|\nabla_xu|^2\,\Delta_x\phi\,\rho_\star}-\irdx{\mathrm{Hess}(\phi):\nabla_xu\otimes\nabla_xu\,\rho_\star}\,.
\end{multline}
To prove~\eqref{ClainedIdentity}, it is indeed enough to notice that
\begin{align*}
\irdx{\(\partial_{x_i}(\rho_\star\,\partial_{x_j} u)\)^2\,\rho_\star^{-1}}&=-\irdx{\partial_{x_i}\phi\,\partial_{x_j}u\,\big(\partial_{x_i}(\rho_\star\,\partial_{x_j}u)\big)}+\irdx{\partial_{x_ix_j}^2u\,\big(\partial_{x_i}(\rho_\star\,\partial_{x_j}u)\big)}\,.
\end{align*}
The observation on the solutions of~\eqref{ell-u} in Section~\ref{Sec:Notation} applies. By integrating by parts, the two integrals are
\begin{align*}
\irdx{\partial_{x_ix_j}^2u\,\big(\partial_{x_i}(\rho_\star\,\partial_{x_j}u)\big)}=&\irdx{\partial_{x_ix_i}^2u\,\big(\partial_{x_j}(\rho_\star\,\partial_{x_j}u)\big)}\,,&\\
 -\irdx{\partial_{x_i}\phi\,\partial_{x_j}u\,\big(\partial_{x_i}(\rho_\star\,\partial_{x_j}u)\big)}=&\irdx{\partial_{x_ix_i}^2\phi\,\partial_{x_j}u\,\big(\rho_\star\,\partial_{x_j}u\big)}+\irdx{\partial_{x_i}\phi\,\partial_{x_ix_j}^2u\,\big(\rho_\star\,\partial_{x_j}u\big)}&\\
 =&\irdx{\partial_{x_ix_i}^2\phi\,\partial_{x_j}u\,\big(\rho_\star\,\partial_{x_j}u\big)}-\irdx{\partial_{x_i}\phi\,\partial_{x_i}u\partial_{x_j}\big(\rho_\star\,\partial_{x_j}u\big)}&\\
 &\hspace*{4.2cm}-\irdx{\partial_{x_jx_i}^2\phi\,\partial_{x_i}u\,\big(\rho_\star\,\partial_{x_j}u\big)}\,.
\end{align*}
Putting everything together, we get
\begin{align*}
 \irdx{\(\partial_{x_i}(\rho_\star\,\partial_{x_j} u)\)^2\,\rho_\star^{-1}}=\irdx{\partial_{x_ix_i}^2\phi\,(\partial_{x_j}u)^2\,\rho_\star}&+\irdx{\partial_{x_i}(\rho_\star\,\partial_{x_i}u)\,\partial_{x_j}\big(\rho_\star\,\partial_{x_j}u\big)}\\
 &-\irdx{\partial_{x_jx_i}^2\phi\,\partial_{x_i}u\,\partial_{x_j}u\,\rho_\star}
\end{align*}
and~\eqref{ClainedIdentity} is obtained by summing over $i$ and $j$. Finally all integrals are estimated by $\bangle{\mathsf{AT}\Pi f,\Pi f}$ using~\eqref{ATPi} and the improved Poincaré inequality~\eqref{ImprPoinc}, which completes the proof of~\eqref{Ellip}.
\\[4pt]
In all cases, we conclude that $\|\mathsf T^2(u\,f_\star)\bangle v^{(1-\beta)_+}\,\|_2^2\le C\,\bangle{\mathsf{AT}\Pi f,\Pi f}$ using~\eqref{ATPi}, for some explicit constant $C>0$. This completes the proof of~\eqref{Ellip}.

\medskip\noindent\textit{\textbf{Conclusion.}} By~\eqref{Estm:TA},~\eqref{ALf} and~\eqref{Ellip}, 
we control $|\bangle{\mathsf{TA} f,f}|$, $\big|\mathrm{Re}\bangle{\mathsf{AL}(1-\Pi)f,f}\big|$ and $\big|\bangle{\mathsf{AT}(1-\Pi)f,\Pi f}\big|$. A discriminant condition on $\delta$ completes the proof of Lemma~\ref{LemmaDissip} as in the proof of Theorem~\ref{ThmAbs:DMS}.\qed\end{proof}

\subsection{Moment estimates}\label{Sec:Moments}

\begin{lemma}\label{Lem:uBounds}
Let $u=u(x)$ be defined in terms of $f$ as in~\eqref{u} and assume For any $k\ge 0$, there exists a constant $C_k>0$ such that
\be{MomentsBound}
M_k:=\irdx{|u|^2\,\bangle x^k\,\rho_\star}\le C_k\irdxv{|\Pi f|^2\,\bangle x^k\,f_\star^{-1}}\,.
\ee
\end{lemma}
\begin{proof}
The case $k=0$ is true from~\eqref{Test1}. By squaring equation~\eqref{u} and testing with $\bangle x^k\,\rho_\star$, we obtain
\[
\irdx{u^2\,\bangle x^k\,\rho_\star}+2\,\sigma\irdx{\nabla_x\big(u\,\bangle x^k\big)\cdot\nabla_xu\,\rho_\star}+\sigma^2\irdx{|\nabla_x(\rho_\star\,\nabla_xu)|^2\,\frac{\bangle x^k}{\rho_\star}}=\irdx{\rho_f^2\,\frac{\bangle x^k}{\rho_\star}}\,.
\]
Moreover, after an integration by parts, we have
\begin{align*}
2\,\sigma\irdx{\nabla_x(u\,\bangle x^k)\cdot\nabla_xu\,\rho_\star}&\,=2\,\sigma\irdx{|\nabla_xu|^2\,\bangle x^k\,\rho_\star}-\sigma\irdx{|u|^2(\Delta_x\bangle x^k-\nabla_x\bangle x^k\cdot\nabla_x\phi)\,\rho_\star}\\
&\,=2\,\sigma\irdx{|\nabla_xu|^2\,\bangle x^k\,\rho_\star}-\sigma\, k\,(k+d-2)\,M_{k-2}+\sigma\, k\,(k-2)\,M_{k-4}\\
&\quad+\sigma\,k\,M_{k+\alpha-2}-\sigma\,k\,M_{k+\alpha -4}\,.
\end{align*}
After dropping the positive terms we get
\[
M_k\le\irdx{\frac{\rho_f^2}{\rho_\star}\,\bangle x^k}+\sigma\,k\,(k+d-2)\,M_{k-2}+\sigma\,k\,M_{k+\alpha-4}\,.
\]
Inequality~\eqref{MomentsBound} follows by induction and interpolation.
\end{proof}
Under the simplifying assumption~\eqref{Linfty}, we also obtain moment estimates directly for the distribution function~$f$. Here there is space for improvements.
\begin{lemma}\label{Lem:UniformBounds} Let $f=f(t,x,v)$ be a solution of~\eqref{kFP} with transport and collision operators given respectively by~\eqref{Tkin} and~\eqref{Lkin} for some $\beta>0$ and $\alpha>0$. Assume that the initial datum $f_0$ satisfies the bound~\eqref{Linfty}. Then for any $k>0$ and for any $\ell>0$ there exist positive constants $C_k$ and $C_\ell$ such that, for any $t\ge0$
\begin{align}
J_k(t):=\irdxv{|f(t,x,v)|^2\,\bangle x^k\,f_\star^{-1}}\le C_k\,,\label{Jk}\\
K_\ell(t):=\irdxv{|f(t,x,v)|^2\,\bangle v^\ell\,f_\star^{-1}}\le C_\ell\,.\label{Kh}
\end{align}
\end{lemma}
\begin{proof}
Since $f_\star$ is a stationary solution, the maximum principle yields 
\[
f(t,\cdot,\cdot)\le C\,f_\star\quad \forall t \ge 0.
\]
Therefore~\eqref{Jk} and~\eqref{Kh} follow by taking 
\[
C_k=C^2\irdxv{f_\star\,\bangle x^k}\quad\mbox{and}\quad C_\ell=C^2\irdxv{f_\star\,\bangle v^\ell}\,.
\]
This completes the proof of Lemma~\ref{Lem:UniformBounds}.
\end{proof}

Notice that, with the elementary estimates
\begin{align*}
&2^{\frac k2-1}\,\big(1+r^k\big)\le\bangle r^k\le1+r^k\quad&\mbox{if}\quad k\in(0,2)\,,\\
&1+r^k\le\bangle r^k\le2^{\frac k2-1}\big(1+r^k\big)\quad&\mbox{if}\quad k\ge2\,,
\end{align*}
we have the simple moment estimate
\begin{multline*}
M_{k,\eta}:=\irdx{\bangle x^k\,e^{-\frac1\eta\,\bangle x^\eta}}\le\max\left\{1,\,2^{\frac k2-1}\right\}\,\big|\mathbb S^{d-1}\big|\int_0^\infty r^{d-1}\,\big(1+r^k\big)\,e^{-\frac{r^\eta}\eta}\,dr\\
=\max\left\{1,\,2^{\frac k2-1}\right\}\,\frac{2\,\pi^\frac d2}{\Gamma(\frac d2)}\,\eta^\frac{d-\eta}\eta\(\Gamma\big(\tfrac d\eta\big)+\eta^\frac k\eta\,\Gamma\big(\tfrac{d+k}\eta\big)\)
\end{multline*}
for any $k>0$ and $\Gamma$ is the Euler Gamma function. As a consequence, $f_\star$ defined by~\eqref{fstar} with $Z=M_{0,\beta}\,M_{0,\alpha}$ is such that
\[
\irdxv{\bangle x^k\,f_\star}=\frac{M_{k,\alpha}}{M_{0,\alpha}}\quad\mbox{and}\quad\irdxv{\bangle v^\ell\,f_\star}=\frac{M_{k,\beta}}{M_{0,\beta}}\,.
\]

\subsection{Proof of \texorpdfstring{Theorem~\ref{Thm:Main}}{Theorem Main}}\label{Sec:ThmMain}

In this section we will work in the framework of Theorem~\ref{Thm:Main}, i.e. we will consider a solution $f=f(t,x,v)$ to the kinetic Fokker-Planck equation~\eqref{kFP} with initial datum $0\le f_0\le C\,f_\star$, for a certain $C>0$. As sign plays no role, up to replacing $f$ with $f-f_\star$ when $f_\star$ is integrable, we may assume that
\be{ZeroMass}
\irdxv{f(t,x,v)}=0\quad\forall\,t\ge 0\,.
\ee
We distinguish various cases depending on the values of $\beta$ and $\alpha$.

\subsubsection*{\texorpdfstring{$\bullet$}{bullet} Case \texorpdfstring{$\beta\ge 1$ and $\alpha\ge 1$}{alpha≥1 and beta≥1}}

Thanks to Lemma~\ref{LemmaDissip}, we have 
\[
\mathsf D[f]\ge\kappa \(\nrmmu{(1-\Pi)f}2^2 + \bangle{\mathsf{AT}\Pi f,\Pi f}\)\,.
\]
for some $\kappa>0$. Because of the assumption $\alpha\ge1$ the operator $(\mathsf T \Pi)^\ast(\mathsf T \Pi)$ is coercive, that is~\eqref{H2} hold. As a consequence we also have~\eqref{gap}, \emph{i.e.}, 
\[
\bangle{\mathsf{AT}\Pi f,\Pi f}\ge \frac{\lambda_M}{1+\lambda_M}\,\nrmmu{\Pi f}2^2\,.
\]
Therefore
\[
\mathsf D[f]\ge \frac{\kappa\,\lambda_M}{1+\lambda_M}\,\nrmmu f2^2
\]
holds for some $\kappa>0$, which gives exponential convergence:
\[
\nrmmu{f(t,\cdot,\cdot)}2^2\le\frac4{2-\delta}\,\mathsf H[f(t,\cdot,\cdot)]\le\frac4{2-\delta}\,\mathsf H[f_0]\,e^{-\lambda t} \quad\mbox{where}\quad \lambda = \frac{\kappa\,\lambda_M}{1+\lambda_M}\,.
\]
Exactly the same proof applies in the case of Corollary~\ref{Cor:DMS}, with $\lambda_M$ now given by the Poincar\'e inequality associated with the measure $e^{-\phi}\,dx$, of which the case $\phi(x)=\frac1\alpha\,\bangle x^\alpha$ with $\alpha\ge1$ is a special case.

\subsubsection*{\texorpdfstring{$\bullet$}{bullet} Case \texorpdfstring{$\beta\in(0,1)$ and $\alpha\ge 1$}{beta in (0,1) and alpha≥1}}

In this case, on the one hand we still have macroscopic coercivity~\eqref{H2} due to the fact that $\alpha\ge 1$, but on the other hand, a loss of weight now appears for the microscopic component because of $\beta\in(0,1)$. Inequality~\eqref{Estimate:Dissipation} now reads as
\[
\mathsf D[f]\ge\kappa \(\|(1-\Pi)f\|^2_{\mathrm L^2(\R^d\times\R^d,\bangle v^{-2(1-\beta)}\,d\mu)} + \frac{\lambda_M}{1+\lambda_M}\,\nrmmu f2^2\)\,.
\]
In order to recover the $\mathrm L^2(\R^d\times\R^d,d\mu)$ norm we need to interpolate with the conservation of moments. Let $\ell>0$ and notice that $f_0\in \mathrm L^2(\R^d\times\R^d,\bangle v^\ell d\mu)$ by our assumption on the initial datum~\eqref{Linfty}. Setting $a=\ell/\big(\ell+2(1-\beta)\big)$, by H\"older's inequality and Lemma~\ref{Lem:UniformBounds}, we have
\begin{align*}
\nrmmu{(1-\Pi)f}2^2 &\le \|(1-\Pi)f\|^{2a}_{\mathrm L^2(\R^d\times\R^d,\bangle v^{-2(1-\beta)}\,d\mu)} \, \|(1-\Pi)f\|^{2(1-a)}_{\mathrm L^2(\R^d\times\R^d,\bangle v^\ell\,d\mu)}\\
&\le \|(1-\Pi)f\|_{\mathrm L^2(\R^d\times\R^d,\bangle v^{-2(1-\beta)}\,d\mu)}^{2a}\,K_\ell(t)^{1-a}\\
&\le C_\ell^{1-a}\,\|(1-\Pi)f\|_{\mathrm L^2(\R^d\times\R^d,\bangle v^{-2(1-\beta)}\,d\mu)}^{2a}\,.
\end{align*}
As a consequence, for a certain constant $C>0$ we have
\[
\mathsf D[f]\ge C\(\nrmmu{(1-\Pi)f}2^{2\(1+\frac{2(1-\beta)}\ell\)}+\nrmmu{\Pi f}2^2\)\ge C\,\nrmmu f2^{2\(1+\frac{2(1-\beta)}\ell\)}\,.
\]
By the Bihari-LaSalle estimate, we finally have
\[
\mathsf H[f]\le\mathsf H[f_0]\(1+C\,\mathsf H[f_0]^{\frac{2(1-\beta)}\ell}\,t\)^{-\frac\ell{2(1-\beta)}}\,.
\]

\subsubsection*{\texorpdfstring{$\bullet$}{bullet} Case \texorpdfstring{$\beta\ge1$ and $\alpha\in(0,1)$}{beta in (0,1) and alpha in (0,1)}}

In this case we have the symmetrical situation compared to the previous one. The dissipation of entropy is 
\[
\mathsf D[f]\ge \kappa\(\nrmmu{(1-\Pi)f}2^2+\bangle{\mathsf{AT}\Pi f,\Pi f}\)
\]
where now $\bangle{\mathsf{AT}\Pi f,\Pi f}$ does not produce macroscopic coercivity, but is given in terms of $u$ by~\eqref{ATPi}. Fix $k>0$ and assume that $f_0\in \mathrm L^2(\R^d\times\R^d,\bangle x^k\,d\mu)$, then from Lemma~\ref{Lem:uBounds} and~\ref{Lem:UniformBounds}, we have that the moments $M_k(t)=\irdx{|u|^2\,\bangle x^k\,\rho_\star}$ are uniformly bounded in time. With $b= k/\big(k+2(1-\alpha)\big)\in(0,1)$, using H\"older's inequality and the weighted Poincaré inequality with non-classical average of~\cite[Cor. 10]{BDLS}, we obtain
\begin{align*}
 \|\Pi f\|_{\mathrm L^2(\R^d\times\R^d,d\mu)}^2&=\irdx{|u|^2\,\rho_\star}+2\,\sigma\irdx{|\nabla_x u|^2\,\rho_\star}+\sigma^2\irdx{|\nabla_x\cdot(\rho_\star\,\nabla_x u)|^2\,\rho_\star}\\
 &\le \(\irdx{|u|^2\,\bangle x^{-2(1-\alpha)}\,\rho_\star}\)^b\(\irdx{|u|^2\,\bangle x^k\,\rho_\star}\)^{1-b}+2\,\bangle{\mathsf{AT}\Pi f,\Pi f}\\
 &\le \big(\mathscr C_\alpha^{\mathrm{wP}}\big)^b\(\irdx{|\nabla_x u|^2\,\rho_\star}\)^b\(\irdx{|u|^2\,\bangle x^k\,\rho_\star}\)^{1-b}+2\,\bangle{\mathsf{AT}\Pi f,\Pi f}\\
 &\le C\,\bangle{\mathsf{AT} \Pi f, \Pi f}^b +2\,\bangle{\mathsf{AT}\Pi f,\Pi f}=:\Phi\big(\bangle{\mathsf{AT} \Pi f, \Pi f}\big)
\end{align*}
where $C$ depends on $\mathscr C_\alpha^{\mathrm{wP}}$ and the bound on $\|f_0\|_{\mathrm L^2(\R^d\times\R^d,\bangle x^k\,d\mu)}$. In the weighted Poincaré inequality we used~\eqref{ZeroMass}, hence $\irdx{u\,\rho_\star}=\irdxv f=0$. Now we have
\begin{align*}
\mathsf H[f]\le \frac{2+\delta}4\,\|f\|^2_{\mathrm L^2(\R^d\times\R^d,d\mu)}&\le \frac{2+\delta}4\(\|(1-\Pi)f\|^2_{\mathrm L^2(\R^d\times\R^d,d\mu)}+ \Phi\big(\bangle{\mathsf{AT} \Pi f, \Pi f}\big) \)\\
&\le \frac{2+\delta}4\,\Phi\(\|(1-\Pi)f\|^2_{\mathrm L^2(\R^d\times\R^d,d\mu)}+ \bangle{\mathsf{AT} \Pi f, \Pi f} \)\\
&\le \frac{2+\delta}4\,\Phi\(\kappa^{-1}\,\mathsf D[f] \)\,.
\end{align*}
Therefore the decay of $\mathsf H[f]$ is estimated by the decay of the solution $z(t)$ of 
\[
z'=\frac{dz}{dt}=-\,\kappa\,\Phi^{-1}\(\frac{4\,z}{2+\delta}\)\,,\quad z(0)=\mathsf H[f_0]\,.
\]
In view of the expression of $\Phi$, we conclude that $z$ monotonically converges to $0$ as $t\to+\infty$ and, as a consequence $z'$ also converges to $0$. This implies that after some time $t_0\ge0$, we have
\[
\Phi\(-\,\kappa^{-1}\,z'\)\le C \(-\,\kappa^{-1}\,z'\)^b\,,
\]
Where $C$ denotes a positive constant that may change from line to line. Altogether, we end up with the differential inequality
\[
z'\le\,-\,C\,z^{1/b}\,.
\]
Integrating and using $\frac{b}{1-b}=\frac k{2(1-\alpha)}$, we obtain that $z(t)\le C\(1+t\)^{-\frac k{2(1-\alpha)}}$.

\subsubsection*{\texorpdfstring{$\bullet$}{bullet} Case \texorpdfstring{$\beta\in(0,1)$ and $\alpha\in(0,1)$}{beta in (0,1) and alpha in (0,1)}}

If $\beta\in(0,1)$ and $\alpha\in(0,1)$ we have neither microscopic coercivity nor macroscopic coercivity. The dissipation of entropy is 
\[
\mathsf D[f]\ge \kappa\(\|(1-\Pi)f\|_{\mathrm L^2(\R^d\times\R^d,\bangle v^{-2(1-\beta)}\,d\mu)}^2+\bangle{\mathsf{AT}\Pi f,\Pi f}\)
\]
and we have to interpolate with moments in both variables $x$ and $v$. As in the previous cases, we have
\[
\nrmmu{\Pi f}2^2\le C\,\bangle{\mathsf{AT} \Pi f, \Pi f}^b +2\,\bangle{\mathsf{AT}\Pi f,\Pi f}= \Phi\big(\bangle{\mathsf{AT}\Pi f,\Pi f}\big)
\]
and
\begin{align*}
\nrmmu{(1-\Pi)f}2^2&\le C\,\|(1-\Pi)f\|_{\mathrm L^2(\R^d\times\R^d,\bangle v^{-2(1-\beta)}\,d\mu)}^{2a}=:\Psi\(\|(1-\Pi)f\|^2_{\mathrm L^2(\R^d\times\R^d,\bangle v^{-2(1-\beta)}\,d\mu) }\)
\end{align*}
with $a=\frac\ell{\ell+2(1-\beta)}$ and $b=\frac k{k+2(1-\alpha)}$. As above we have
\begin{align*}
\mathsf H[f] &\le \frac{2+\delta}4\(\Psi\(\|(1-\Pi)f\|^2_{\mathrm L^2(\R^d\times\R^d,\bangle v^{-2(1-\beta)}\,d\mu)}\)+ \Phi\big(\bangle{\mathsf{AT} \Pi f, \Pi f} \big) \)\\
&\le \frac{2+\delta}4\(\Psi\big(\kappa^{-1}\,\mathsf D[f]\big)+ \Phi\big(\kappa^{-1}\,\mathsf D[f]\big) \)\,.
\end{align*}
Notice that the function $t\mapsto\Psi(t)+\Phi(t)$ is increasing, concave and $\Psi(0)+\Phi(0)=0$. Moreover we have
\[
\frac d{dt}\mathsf H[f(t,\cdot,\cdot)]\le -\,\kappa\,(\Psi+\Phi)^{-1}\(\frac4{2+\delta}\,\mathsf H[f(t,\cdot,\cdot)]\)\,.
\]
As a consequence, $\mathsf H[f(t,\cdot,\cdot)]$ can be estimated by the solution $z$ of
\[
z'=-\,\kappa\,(\Psi+\Phi)^{-1}\(\frac{4\,z}{2+\delta}\)\,.
\]
For the same reasons as before, $z'$ converges to $0$ as $t\to+\infty$. Using the explicit expressions of $\Phi$ and $\Psi$, we see that there exists some $t_0\ge0$ such that, for any $t\ge t_0$,
\[
(\Psi+\Phi)\(-\,\kappa^{-1}\,z'\)\le C\(-\,\kappa^{-1}\,z'\)^\zeta
\]
for some $C>0$, where $\zeta=\min\{a,b\}$. This inequality leads to $z'\le-\,C\,z^{1/\zeta}$ and therefore to
\[
z(t)\le C\,(1+t)^{-\min\big\{\frac k{2(1-\alpha)},\frac\ell{2(1-\beta)}\big\}}\quad\forall\,t\ge0
\]
by the Bihari-LaSalle estimate.

\subsubsection*{\texorpdfstring{$\bullet$}{bullet} Case \texorpdfstring{$\beta\ge 1$ and $\phi=0$}{beta≥1 and alpha=0}}

In absence of a global equilibrium, we can still consider~\eqref{Estimate:Dissipation} written with $\rho_\star=1$. Identity~\eqref{ATPi} now reads as 
\[
\bangle{\mathsf{AT}\Pi f,\Pi f}=\sigma\irdx{|\nabla_xu|^2}+\sigma^2\irdx{|\Delta_xu|^2}\,.
\]
Because of Nash's inequality and the conservation of mass, we have
\begin{align*}
\|\Pi f\|_{\mathrm L^2(\R^d\times\R^d,d\mu)}^2&=\|u\|^2_{\mathrm L^2(dx)}+2\,\sigma\|\nabla_xu\|^2_{\mathrm L^2(dx)}+\sigma^2\|\Delta_x u\|^2_{\mathrm L^2(dx)}\\
&\le C_{\mathrm{Nash}}\,\|u\|_{\mathrm L^1(dx)}^{\frac4{d+2}}\,\|\nabla_xu\|_{\mathrm L^2(dx)}^{\frac{2d}{d+2}}+2\,\bangle{\mathsf{AT}\Pi f,\Pi f}\\
&\le C\,\bangle{\mathsf{AT}\Pi f,\Pi f}^{\frac d{d+2}}+2\,\bangle{\mathsf{AT}\Pi f,\Pi f}\,.
\end{align*}
In the asymptotic regime of small $\nrmmu{\Pi f}2^2$, we have
\[
\bangle{\mathsf{AT}\Pi f,\Pi f}\ge C\,\|\Pi f\|^{2\(1+\frac2d\)}
\]
for some suitable constant $C>0$, and
\[
 \mathsf D[f]\ge C \(\|(1-\Pi)f\|_{\mathrm L^2(\R^d\times\R^d,d\mu)}^2 + \|\Pi f\|^{2\(1+\frac2d\)}_{\mathrm L^2(\R^d\times\R^d,d\mu)}\)
 \ge C\,\|f\|^{2\(1+\frac2d\)}\,.
\]
By the Bihari-LaSalle estimate, we can finally conclude 
\[
\mathsf H[f(t,\cdot,\cdot)]\le\mathsf H[f_0]\(1+C\,\mathsf H[f_0]^{\frac2d}\,t\)^{-\frac d2}\forall\,t\ge0\,.
\]

\subsubsection*{\texorpdfstring{$\bullet$}{bullet} Case \texorpdfstring{$\beta\in(0,1)$ and $\phi=0$}{beta in (0,1) and alpha=0}}

We proceed as in the previous case. The macroscopic part obeys the same estimate
\[
\bangle{\mathsf{AT}\Pi f,\Pi f}\gtrsim C\,\|\Pi f\|^{2\(1+\frac2d\)}\,.
\]
The microscopic component has to be interpolated with moments:
\[
\nrmmu{(1-\Pi)f}2^2 \le C\,\|(1-\Pi)f\|_{\mathrm L^2(\R^d\times\R^d,\bangle v^{-2(1-\beta)}\,d\mu)}^{2b}\,.
\]
We conclude that, as $\nrmmu f2\to 0$,
\[
\mathsf D[f]\ge C\(\nrmmu{(1-\Pi)f}2^{2\(1+\frac{2(1-\beta)}k\)} + \nrmmu{\Pi f}2^{2\(1+\frac2d\)}\)\ge C\,\nrmmu f2^{2\(1+\frac1\zeta\)}
\]
where $\zeta=\min\{\frac k{2(1-\beta)},\frac d2\}$. By the Bihari-LaSalle estimate we conclude
\[
\mathsf H[f]\le\mathsf H[f_0]\(1+C\,\mathsf H[f_0]^\frac1\zeta\,t\)^{-\zeta}\,.
\]

\begin{acknowledgement}
L.Z.~received funding from the European Union’s Horizon 2020 research and innovation program under the Marie Skłodowska-Curie grant agreement No.~945322. J.D.~thanks S.~Menozzi, A.~Pascucci and S.~Polidoro for the organization of the meeting on \emph{Kolmogorov operators and their applications} which has turned to be a great opportunity to write down and collect the results of the present contribution. All authors thank an anonymous referee for his/her valuable suggestions and remarks.
\par\smallskip\noindent{\scriptsize\copyright\,\the\year\ by the authors. This paper may be reproduced, in its entirety, for non-commercial purposes. \href{https://creativecommons.org/licenses/by/4.0/legalcode}{CC-BY 4.0}}
\end{acknowledgement}
\newpage


\end{document}